\documentclass[12pt]{amsart}
\usepackage{amssymb,amsmath,amsthm,amsfonts,mathrsfs}
\usepackage{geometry}
\usepackage[utf8]{inputenc}
\usepackage{amsmath, amsthm, amsfonts, amssymb,epsfig,enumerate,
  wrapfig, scalerel, tikz,  color, accents, fancybox, rotating,
  comment, afterpage, changepage, colortbl, float, xlop, blkarray, hyperref, cleveref, setspace, tikz-3dplot, pgf}
\title[A topological theory of unoriented \texorpdfstring{$SL(4)$}{SL(4)} foams]{A topological theory for unoriented \texorpdfstring{$SL(4)$}{SL(4)} foams}
\author[M. Khovanov]{Mikhail Khovanov}
\address{John Hopkins University, US \& Columbia University, US}
\email{\href{mailto:khovanov@math.columbia.edu}{khovanov@math.columbia.edu}}
\author[J. H. Przytycki]{J\'ozef H. Przytycki}
\address{George Washington University, US \& University of Gdańsk, Poland}
\email{\href{mailto:przytyck@gwu.edu}{przytyck@gwu.edu}}
\author[L.-H. Robert]{Louis-Hadrien Robert}
\address{Université Clermont Auvergne, France}
\email{\href{mailto:louis\_hadrien.robert@uca.fr}{louis\_hadrien.robert@uca.fr}}
\author[M. Silvero]{Marithania Silvero}
\address{Universidad de Sevilla, Spain}
\email{\href{mailto:marithania@us.es}{marithania@us.es}}
\makeatletter
\@namedef{subjclassname@2020}{\textup{2020} Mathematics Subject Classification}
\makeatother
\hypersetup{
    pdftoolbar=true,        
    pdfmenubar=true,        
    pdffitwindow=false,     
    pdfstartview={FitH},    
    pdftitle={\shorttitle},    
    pdfauthor={\shortauthors},     
    pdfsubject={\shorttitle},   
    pdfcreator={\shortauthors},   
    pdfproducer={\shortauthors}, 
    pdfkeywords={}, 
    pdfnewwindow=true,      
    colorlinks=true,       
    linkcolor=darkgray,          
    citecolor=teal,        
    filecolor=magenta,      
    urlcolor=violet,          
    linkbordercolor=red,
    citebordercolor=teal,
    urlbordercolor=violet,  
    linktocpage=true
    }

\usetikzlibrary{math, arrows, decorations.markings, decorations, patterns, positioning, decorations.pathreplacing,
  decorations.pathmorphing, decorations.text, decorations.shapes}
\tikzset{->-/.style={decoration={markings,
  mark=at position .5 with {\arrow{>}}},postaction={decorate}}}
\tikzset{-<-/.style={decoration={markings, mark=at position .5 with {\arrow{<}}},postaction={decorate}}}

\let\oldtocsection=\tocsection
\let\oldtocsubsection=\tocsubsection
\renewcommand{\tocsection}[2]{\hspace{0em}\oldtocsection{#1}{#2}}
\renewcommand{\tocsubsection}[2]{\hspace{1em}\oldtocsubsection{#1}{#2}}

\newcounter{res}[section]
\numberwithin{res}{section}
\newtheorem{theorem}[res]{Theorem}

\newtheorem{lemma}[res]{Lemma}
\newtheorem{proposition}[res]{Proposition}

\newtheorem{question}[res]{Question}
\theoremstyle{definition}

\newtheorem{definition}[res]{Definition}
\newtheorem{remark}[res]{Remark}

\newtheorem{example}[res]{Example}

\newcommand{\R}{\ensuremath{\mathbb{R}}}
\newcommand{\Z}{\ensuremath{\mathbb{Z}}}
\newcommand{\N}{\ensuremath{\mathbb{N}}}
\newcommand{\C}{\ensuremath{\mathbb{C}}}
\newcommand{\ftwo}{\ensuremath{\mathbb{F}_2}}
\renewcommand{\SS}{\ensuremath{\mathbb{S}}}

\newcommand{\tc}{t_4}
\newcommand{\ncol}{\tc}
\newcommand{\defect}{\ensuremath{\operatorname{dft}}}
\newcommand{\seam}{\ensuremath{s}}
\newcommand{\thickness}{\ensuremath{\operatorname{th}}}
\newcommand{\rank}{\ensuremath{\operatorname{rk}}}
\newcommand{\colorings}{\ensuremath{\operatorname{col}}}

\newcommand{\pigments}{\ensuremath{\mathbb{P}}}
\newcommand{\ourring}{\ensuremath{R}}
\newcommand{\localring}{\ourring_{\discriminant}}
\newcommand{\discriminant}{\delta}
\newcommand{\basering}{\ensuremath{\ftwo[X_1, X_2, X_3, X_4]^{S_4}}}

\newcommand{\YD}{\operatorname{{YD}}}
\makeatletter
\newcommand{\eval}[1][\relax]{\ifx#1\relax \ensuremath{\tau} \else \ensuremath{\tau\left(#1\right)} \fi}
\newcommand{\functor}[1][\relax]{\ifx#1\relax \ensuremath{\tau} \else \ensuremath{\tau\left(#1\right)} \fi}
\newcommand{\functord}[1][\relax]{\ifx#1\relax \ensuremath{\tau_{\delta}} \else \ensuremath{\tau_{\delta}\left(#1\right)} \fi}
\newcommand{\forget}[1][\relax]{\ifx#1\relax \ensuremath{\mathtt{f}} \else \ensuremath{\mathtt{f}\left(#1\right)} \fi}
\makeatother

\newcommand{\myvertextwo}{\NB{\tikz{\filldraw[color=red,fill=white](0,0) circle (0.5mm);}}}
\newcommand{\myvertex}{\NB{\tikz{\filldraw[fill=white](0,0) circle (0.5mm);}}}
\newcommand{\kempesquare}{\ensuremath{\diamond}}
\newcommand{\kempetriangle}{\ensuremath{\triangle}}

\newcommand{\imagesfolder}{.}
\newcommand{\NB}[1]{\ensuremath{\vcenter{\hbox{#1}}}}

\newcommand{\catFoam}{\ensuremath{\mathsf{Foam}}}

\tikzset{set style/.style={#1},
  expand style/.style={set style/.expanded={#1}},
  expand style once/.style={set style/.expand once={#1}},
  expand style twice/.style={set style/.expand twice={#1}}
}

 \tikzset{zxplane/.style={canvas is zx plane at y=#1}}
 \tikzset{yxplane/.style={canvas is yx plane at z=#1}}
 \tikzset{zyplane/.style={canvas is zy plane at x=#1}}

\definecolor{colork}{RGB}{255,127,127}
\definecolor{colorj}{RGB}{30,30,255}
\definecolor{colorl}{RGB}{0,127,0}
\definecolor{colori}{RGB}{255,159,31}

\newcommand{\stylecoli}{colori, thick, decorate, decoration={snake, segment length=1mm,amplitude=0.2mm}}
\newcommand{\stylecolj}{colorj,very thick, densely dashed}
\newcommand{\stylecolk}{colork, decorate, thin, decoration={coil,
    segment length=0.5mm,amplitude=0.3mm}}
\newcommand{\stylecoll}{colorl, decorate, thin, decoration={crosses,segment
    length=1mm, shape width= 0.8mm, shape height=0.8mm}}

\newcommand{\colori}{{\color{colori}1}}

\let\emptyset\varnothing

\newcommand{\lra}{\longrightarrow}

\newcommand{\mcP}{\mathcal{P}}

\newcommand{\regions}{r} 
\newcommand{\id}{\mathsf{id}}
\newcommand{\Webs}{\mathsf{Webs}}
\newcommand{\kk}{\mathbf{k}}

\begin{document}
\begin{abstract} Unoriented SL(3) foams are two-dimensional CW complexes with generic singularities embedded in  3- and 4-manifolds. They naturally come up in the Kronheimer--Mrowka SO(3) gauge theory for 3-orbifolds and, in the  oriented case, in a categorification of the Kuperberg bracket quantum invariant.  The present paper studies the  more technically complicated case of SL(4) foams. Combinatorial evaluation of unoriented SL(4) foams is defined and state spaces for it are studied. In particular, over a suitably localized ground ring, the state space of any web is free of the rank given by the number of its 4-colorings.  
\end{abstract}

\date{July 2, 2023}

\maketitle
\tableofcontents

\section{Introduction}
\label{sec_intro}
Foams, as they appear in link homology, are suitably decorated
combinatorial two-di\-men\-sion\-al CW-complexes embedded in
$\R^3$. Combinatorial approach to $GL(N)$ link homology, see \cite{Kho04,
  MacVaz, MacStoVaz, QueRos, RobWag} and other papers, is based on a
special evaluation of closed foams (foams without boundary) to
elements of rings of symmetric functions $R_N$ in $N$ variables. One
then uses so-called \emph{universal construction} of topological
theories \cite{BHMV, KhovUniversal} to build a lax TQFT from foam
evaluation. This lax TQFT (or a topological theory) defines state
spaces of plane graphs, which appear as generic plane cross-sections
of foams. Given a plane diagram of a link, one ``resolves'' it into a
collection of plane graphs and assigns a complex of state spaces, one
space for each graph in the resolution, with differentials coming from
maps induced by cobordisms in the topological theory. Homology of this
complex categorifies the Reshetikhin--Turaev--Witten quantum link
invariant with components colored by fundamental representations of
quantum $GL(N)$. There are now many categorical approaches to this link homology; for instance, one of the earliest ones is based on matrix factorizations \cite{KhoRozI, Yonezawa, Wu}. 

Other approaches are based on representation theory
\cite{BernsteinFrenkelKhovanov, SussanThesis, Webster, StroppelSussan}, coherent sheaves on quiver varieties \cite{CautisKamnitzer, AnnoNandakumar}, Fukaya--Floer categories on quiver varieties \cite{SeidelSmith, Manolescu, AbouzaidSmith}, and mathematical physics \cite{GukovScwarzVafa, Witten, Aganagic}. Foam approach is distinguished by being the most combinatorial, with relevant categories and functors that control the theory appearing much later in the construction. 

Spines are two-dimensional CW-complexes with generic singularities, strongly reminiscent of foams. They naturally come up in the simple homotopy theory of 2-dimensional CW-complexes and in 3-manifold theory (as 2-skeleta of the Poincaré dual of a triangulation of a 3-manifold). 

A variation on $GL(N)$-foams and their evaluations emerges in gauge
theory, via Kron\-hei\-mer--Mrow\-ka homology theory for 3-orbifolds
\cite{KM3, KM2, KM1, KhovRob1}. The corresponding
foams are sometimes called \emph{unoriented $SL(3)$ foams}
\cite{KhovRob1, Boozer1,KhovRob2, Boozer2}. 
Kron\-heim\-er--Mrow\-ka theory and unoriented $SL(3)$ foams have a close
connection to the Four-Color Theorem, with the potential to give a
conceptual proof of the latter and rethink it via gauge theory and 3-
and 4-dimensional topology. Unoriented $SL(3)$ foams and webs are
substantially more complicated than their oriented counterparts, and
beyond certain reducible cases state spaces of unoriented webs have
not been determined.

The present paper studies unoriented $SL(4)$ foams and webs. 
Unoriented $SL(4)$ webs relate to 4-regular planar graphs and their
four colorings. Colorability of 4-regular planar graph is
well-understood \cite{Guenin, CEKS}. However the number of
such colorings does not seem to have easy reduction local formula for
small facets. For instance, there is no simplification for
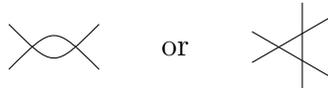
\[
\NB{\tikz[]{\begin{scope}[scale=1.2]
  \draw (0,0) .. controls +(0.5,0.5) and +(-0.5, 0.5) .. (1,0);
  \draw (0,0.5) .. controls +(0.5,-0.5) and +(-0.5, -0.5) .. (1,0.5); 
\end{scope}}} \qquad \text{or} \qquad \NB{\tikz[]{\begin{scope}[scale=1.2]
  \draw (40 :0.5) -- (200:0.5);
  \draw (80 :0.5) -- (280:0.5);
  \draw (160 :0.5) -- (320:0.5);
\end{scope}}}
\]
Hence it makes sense to change the setup a little bit and to replace
a 4 valent-vertex:  
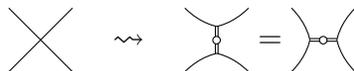
\[
  \NB{\tikz[]{\begin{scope}[scale=1.2]
  \draw (45:0.5) -- (225:0.5);
  \draw (135:0.5) -- (315:0.5);
\end{scope}}} \ \ \ \rightsquigarrow \ \ \ \NB{\tikz[scale=1.2]{\input{\imagesfolder/4v_dumbell-bullet}}} = \NB{\tikz[scale=1.2]{\input{\imagesfolder/4v_dumbell-bullet-intro}}}
\]
imposing a flow-like condition on trivalent vertices and a
complementary condition on bivalent vertices. This is precisely our
definition of colorings of unoriented $SL(4)$ webs, see
Definitions~\ref{def:web} and \ref{def:weboloring}.

We introduce unoriented $SL(4)$ webs (just called \emph{webs}) and
unoriented $SL(4)$ foams as cobordisms between webs. Fundamental
$SL(4)$ representation $V$ generates nontrivial exterior powers
$\Lambda^2 V$ and $\Lambda^3 V$, which can translate into foams having
corresponding facets of \emph{thickness} two and three. To make the presentation more compact, we do not consider facets of thickness three, which can be avoided. Facets corresponding to $V$ are called \emph{thin} and those corresponding to $\Lambda^2 V$ are called \emph{thick}.   

The basic evaluation function on closed unoriented $SL(4)$ webs counts the number of Tait colorings $t_4(\Gamma)$ of edges of a web $\Gamma$. We write down some skein relations for that function in Section~\ref{sec_webs}

In Section~\ref{sec:foams-and-eval} we define (unoriented $SL(4)$) foams $F$ and their evaluations $\functor(F)$ to elements of the ring 
$R$ of symmetric functions in four variables over the 2-element field $\mathbb{F}_2$. Characteristic two reduction is forced upon us by lack of facet orientations of our foams, just as in the $SL(3)$ case~\cite{KhovRob1}. The notion of coloring of a foam is introduced, with thin (respectively thick) facets labelled by a single pigment  (respectively a pair of pigments) from the set $\pigments=\{1,2,3,4 \}$of four pigments, subject to compatibility conditions at seams and vertices of the foam. Additionally, foams can carry defect lines at thick facets, and a coloring needs to be reversed at a defect line. In Section~\ref{sec:neck-cutt-relat} we write down a number of skein relations on foams. 

In Section~\ref{sec_state_sp} state spaces of webs are defined via the universal construction. The state space $\functor(\Gamma)$ of a web $\Gamma$ is a graded $R$-module generated by all foams $F$ with boundary $\Gamma$ modulo relations obtained by closing the foams in all possible ways and evaluating closing foams. This approach to state spaces is known as the \emph{universal construction to topological theories}, see earlier references. It results in a lax monoidal functor from the category of foams to the category of graded $R$-modules, see Theorem~\ref{thm_functor}. 

Skein relations from Section~\ref{sec:neck-cutt-relat} give reduction (or direct sum decomposition) formulas for state spaces of webs that contains a region with at most three sides or a region with four sides subject to some constraints, as explained in Theorem~\ref{thm_dir_sum_rels}. We then introduce the notion of \emph{reducible} web and show that its state space $\functor(\Gamma)$ is a free graded $R$-module of rank equal to the number of Tait colorings of $\Gamma$. The graded rank of the state space of a reducible web can be computed as well, giving a quantum-flavor invariant of such webs. For non-reducible webs this invariant can be defined as well, but at present we lack the tools to determine it (and to compute the state space of any non-reducible web). An example of a non-reducible web is given in that section.

Foam evaluation and state spaces are defined over the ring $R$ of symmetric polynomials in four variables $X_1,X_2,X_3,X_4$ with coefficients in $\mathbb{F}_2$. This ring has elementary symmetric polynomials as generators, $R\cong \mathbb{F}_2[E_1,\dots, E_4].$ Consider the discriminant $\delta^2$ of the polynomial 
\[
X^4+E_1X^3+E_2 X^2+E_3 X + E_4 = \prod_{i=1}^4 (X+X_i). 
\]
One can pass to the larger ring $R_{\delta}$ where $\delta^2$
(equivalently, $\delta$), is invertible. Inverting $\delta$ takes most
of complexity out of the state spaces associated with webs
$\Gamma$. In particular, over $R_{\delta}$, we are able to lift
Tutte-like relations \eqref{eq_tutte_1}--\eqref{eq_tutte_3} on the
$t_4$ invariant of various webs in Proposition~\ref{prop_tutte} to
direct sum decompositions of state spaces, see
Proposition~\ref{prop_functd}. This implies that $R_{\delta}$ state
spaces $\functor_{\delta}(\Gamma)$ are free $R_{\delta}$-modules of
rank $t_4(\Gamma)$. However, information on the $\Z$-grading of state
spaces over $R$ is lost when passing to $R_{\delta}$, due to $\delta$
as well as its factors $X_i+X_j$ being invertible and inducing
isomorphisms of homogeneous terms of $\functor_{\delta}(\Gamma)$.

\vspace{0.07in} 

$SL(3)$ and $SL(4)$ unoriented foam theories are similar and share the same foundational unsolved problems, including understanding state spaces of nonreducible planar webs. 
Some open questions are listed next:  
\begin{itemize}
    \item Develop techniques to compute state spaces of non-reducible unoriented $SL(3)$ and $SL(4)$ webs and understand the corresponding quantum invariant of these webs.  
    \item Is there a connection between  unoriented $SL(4)$ webs and gauge theory similar to the one in~\cite{KM1,KhovRob1} for unoriented $SL(3)$ webs and gauge theory? 
    \item Is there a relation between evaluations of unoriented $SL(4)$ foams with thick facets only and $SL(3)$ foams coming from the surjection of symmetric groups $S_4\lra S_3$ given by the permutation action of $S_4$ on the pairs of pigments in $\{1,2,3,4\}$ with opposite pairs identified?  
\end{itemize}

\vspace{0.07in} 

{\bf Acknowledgments:} A significant part of this work was done in
June 2023 at the \emph{Research in Pairs} program at the Oberwolfach Mathematical Institute (MFO). The authors are very grateful to the MFO for this opportunity and would like to acknowledge their indispensable support in producing the present
paper. M.~K.{} was partially supported by NSF grant DMS-2204033 and Simons Collaboration Award 994328.  J.H.~P.{} was partially supported by Simons Collaboration Award 637794.  L.-H.~R.{} thanks the LMBP and in particular Sylvie Chassagne and Valérie Sourlier for their logistical support. M.~S.{} was partially supported by Spanish Research Project PID2020-117971GB-C21, by IJC2019-040519-I, funded by MCIN/AEI/10.13039/501100011033 and by SI3/PJI/2021-00505 (Comunidad de Madrid).

\section{Webs and 4-colorings}
\label{sec_webs}

To work with the category of (unoriented $SL(4)$) foams, we start by considering (unoriented $SL(4)$) webs, which are suitable decorated planar graphs.  

\begin{definition}\label{def:web}
    A \emph{web} is an unoriented plane graph $\Gamma$ with a label in $\{1,2\}$, called \emph{thickness}, associated to each edge. We write $E(\Gamma)$ for the set of edges of $\Gamma$, and say that an edge labelled $i\in\{1,2\}$ is an \emph{$i$-edge}. Vertices of $\Gamma$ can be of three types:
    \begin{itemize}
    \item[(a)] A trivalent vertex adjacent to three $2$-edges (a $222$-vertex).\item[(b)] A trivalent vertex adjacent to a $2$-edge and two
      $1$-edges (a $112$-vertex).\item[(c)] A bivalent vertex, that we call  \emph{defect}, adjacent to two $2$-edges. 
    \end{itemize}
\end{definition}

\begin{figure}
    \centering
\NB{\tikz[]{\input{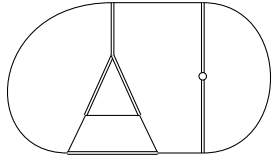}}}
\caption{\small{A web with one $222$-vertex, seven $112$-vertices and a defect vertex.}} 
\label{fig:wedge}
\end{figure}

\begin{definition} \label{def:weboloring}
Define a \emph{pigment} to be an element of the set $\pigments=\{1,2,3,4\}$. 
We call $\pigments$ the set of pigments and define $\mcP(\pigments)$ to be the set of all subsets of $\pigments.$

A \emph{coloring} of a web $\Gamma$ is a map
\[ c: E(\Gamma) \to \mcP(\pigments) \]
associating to each $i$-edge a subset of cardinality $i$ of $\pigments=\{1,2,3,4\}$. We require the following conditions on the coloring of adjacent edges:
\begin{enumerate}
\item[(a)] At each $112$-vertex the coloring of the $2$-edge is the union of colorings of the two $1$-edges. 
\item[(b)] At each $222$-vertex each of the colors of the  adjacent
  $2$-edges are the three two-element subsets of a cardinality three subset of $\pigments$.  
\item[(c)] At each defect the colorings of the adjacent $2$-edges are complementary. 
\end{enumerate}
\begin{figure}[ht]
  \centering
  \NB{\tikz[scale=0.6]{\input{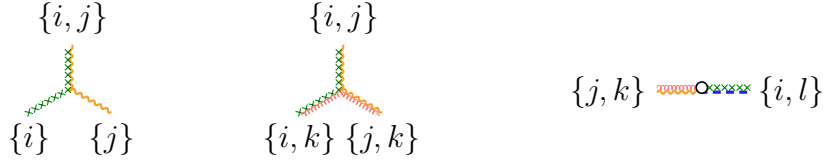}}}
  \caption{From left to right: prototypical colorings at
    $112$-vertices, $222$-vertices and defects.}
  \label{fig:col-vertex}
\end{figure}
Prototypical colorings at vertices are depicted on Figure~\ref{fig:col-vertex}, with $\pigments= \{i,j,k,l\}$.
We write $\tc(\Gamma)$ for the number of colorings of $\Gamma$ and call it \emph{the Tait number} of $\Gamma$. The web $\Gamma_1$ in Figure \ref{fig:wedge} admits no colorings, and therefore $\tc(\Gamma_1)=0$. The empty web $\emptyset_1$ admits a single coloring. 
\end{definition}

\begin{proposition}\label{prop:relations}
The function $t_4$ is multiplicative with respect to disjoint unions of
webs. In particular, $\ncol(\emptyset_1) =1$. Moreover, the following local relations hold:
\begin{align}
    \label{eq_add_circle} \ncol\left(\NB{\tikz[]{\begin{scope}
  \draw (0,0) circle (0.2cm);
\end{scope}}} \right) &= 4  , &&& 
    \ncol\left(\NB{\tikz[]{\begin{scope}
  \draw[double] (0,0) circle (0.2cm);
\end{scope}}} \right) &= 6  , \\
\label{eq_loop_zero}
    \ncol\left(\NB{\tikz[]{\begin{scope}
\clip (-0.3,-0.22) rectangle (0.5, 0.22); 
  \draw[double] (-0.3,0) -- (0,0);
  \draw (0,0) .. controls +(45:0.8) and + (-45:0.8) .. (0,0);
\end{scope}}}\right) &= 0, &&&
\ncol\left(\NB{\tikz[]{\begin{scope}
  \draw[double] (0,0) circle (0.3cm);
  \node at (0.3,0) {$\myvertex$};
\end{scope}}}\right) &= 0, \\ 
\label{eq_defects}
    \ncol\left(\NB{\tikz[]{\begin{scope}
  \draw[double] (-0.5,0) -- (0.5,0) coordinate[pos=0.3] (v1)
  coordinate[pos=0.7] (v2);
  \node[anchor=center, yshift=-0.15mm] at (v1) {$\myvertex$};
  \node[anchor=center, yshift=-0.15mm] at (v2) {$\myvertex$};
\end{scope}
}}\right) &=
    \ncol\left(\NB{\tikz[]{\begin{scope}
  \draw[double] (-0.3,0) -- (0.3,0);
\end{scope}
}}\right), &&& \ncol\left(\NB{\tikz[]{ \begin{scope}
    \draw[double] (0,0) -- (90: 0.5);
    \node at (90: 0.25) {$\myvertex$};
    \draw[double] (0,0) -- (210: 0.5);
    \draw[double] (0,0) -- (-30: 0.5); 
  \end{scope}}}\right) &= \ncol\left(\NB{\tikz[]{ \begin{scope}
    \draw[double] (0,0) -- (90: 0.5);
    \draw[double] (0,0) -- (210: 0.5);
    \draw[double] (0,0) -- (-30: 0.5);
    \node at (-150: 0.25) {$\myvertex$};
  \end{scope}}}\right), \\
\label{eq_digon_1}
    \ncol\left(\NB{\tikz[]{\begin{scope}
  \draw[double] (-0.5,0) -- (-0.3,0);
  \draw[double] ( 0.5,0) -- ( 0.3,0);
  \draw (-0.3,0) .. controls +(60:0.3) and + (120:0.3) .. (0.3,0);
  \draw (-0.3,0) .. controls +(-60:0.3) and + (-120:0.3) .. (0.3,0);
\end{scope}
}}\right) &=
2\, \ncol\left(\NB{\tikz[]{}}\right), &&&
    \ncol\left(\NB{\tikz[]{\begin{scope}
  \draw[double] (-0.5,0) -- (-0.3,0);
  \draw[double] ( 0.5,0) -- ( 0.3,0);
  \draw[double] (-0.3,0) .. controls +(60:0.3) and + (120:0.3) .. (0.3,0);
  \draw[double] (-0.3,0) .. controls +(-60:0.3) and + (-120:0.3) .. (0.3,0);
\end{scope}}}\right) &=
    4\, \ncol\left(\NB{\tikz[]{}}\right), \\
\label{eq_digon_2}
    \ncol\left(\NB{\tikz[]{\begin{scope}
  \draw (-0.5,0) -- (-0.3,0);
  \draw ( 0.5,0) -- ( 0.3,0);
  \draw[double] (-0.3,0) .. controls +(60:0.3) and + (120:0.3) .. (0.3,0);
  \draw (-0.3,0) .. controls +(-60:0.3) and + (-120:0.3) .. (0.3,0);
\end{scope}
}}\right) &=
    3\, \ncol\left(\NB{\tikz[]{\begin{scope}
  \draw (-0.3,0) -- (0.3,0);
\end{scope}
}}\right), &&& \ncol\left(\NB{\tikz[]{\begin{scope}
  \draw (-0.5,0) -- (-0.3,0);
  \draw ( 0.5,0) -- ( 0.3,0);
  \draw[double] (-0.3,0) .. controls +(60:0.3) and + (120:0.3) .. (0.3,0);
\node at (0,0.18) {$\myvertex$};
  \draw (-0.3,0) .. controls +(-60:0.3) and + (-120:0.3) .. (0.3,0);
\end{scope}}}\right) &=
    0, 
    \\
\label{eq_quadruple_v}\ncol\left( \NB{\tikz[]{\input{\imagesfolder/4v_dumbell-bullet}}}\right) &=
    \ncol\left(\NB{\tikz[]{\input{\imagesfolder/4v_dumbell-bullet-intro}}}\right) \\
\label{eq_triangle_1}
\ncol\left(\NB{\tikz[]{\begin{scope}
    \draw[double] (90:0.25) -- (90: 0.5); 
    \draw (-150:0.25) -- (-30: 0.25); 
    \draw (-150:0.25) -- (90: 0.25);
    \draw (-30:0.25) -- (90: 0.25); 
    \draw[double] (-150:0.25) -- (-150: 0.5);
    \draw[double] (-30:0.25) -- (-30: 0.5); 
  \end{scope}}}\right) & = \ncol\left(\NB{\tikz[]{\begin{scope}
    \draw[double] (0,0) -- (90: 0.5); 
    \draw[double] (0,0) -- (210: 0.5) coordinate [pos=0.5] (v); 
    \draw[double] (0,0) -- (-30: 0.5);
\end{scope}

}}\right),   &&& \ncol\left(\NB{\tikz[]{\begin{scope}
    \draw[double] (90:0.25) -- (90: 0.5); 
    \draw[double] (-150:0.25) -- (-30: 0.25); 
    \draw[double] (-150:0.25) -- (90: 0.25);
    \draw[double] (-30:0.25) -- (90: 0.25); 
    \draw[double] (-150:0.25) -- (-150: 0.5);
    \draw[double] (-30:0.25) -- (-30: 0.5); 
  \end{scope}}}\right) & = 2\, \ncol\left(\NB{\tikz[]{}}\right),   \\
\label{eq_triangle_2}
\ncol\left(\NB{\tikz[]{\begin{scope}
    \draw[double] (90:0.25) -- (90: 0.5); 
    \draw[double] (-150:0.25) -- (-30: 0.25); 
    \draw (-150:0.25) -- (90: 0.25);
    \draw (-30:0.25) -- (90: 0.25); 
    \draw (-150:0.25) -- (-150: 0.5);
    \draw (-30:0.25) -- (-30: 0.5); 
  \end{scope}

}}\right) &= \ncol\left(\NB{\tikz[]{\begin{scope}[rotate = -120]
    \draw (0,0) -- (90: 0.5); 
    \draw[double] (0,0) -- (210: 0.5) coordinate [pos=0.5] (v); 
    \draw (0,0) -- (-30: 0.5);
\end{scope}

}}\right),   &&&  \ncol\left(\NB{\tikz[]{\begin{scope}
    \draw[double] (90:0.25) -- (90: 0.5); 
    \draw (-150:0.25) -- (-30: 0.25); 
    \draw[double] (-150:0.25) -- (90: 0.25);
    \draw[double] (-30:0.25) -- (90: 0.25); 
    \draw (-150:0.25) -- (-150: 0.5);
    \draw (-30:0.25) -- (-30: 0.5); 
  \end{scope}}}\right) &= 2\, \ncol\left(\NB{\tikz[]{}}\right),  
\end{align}\vspace{-0.35cm}
\begin{align}
\label{eq_square_1}
\ncol\left( \NB{\tikz[]{\begin{scope}
\begin{scope}
  \draw[double] (45:0.3) -- (45:0.5);
  \draw[double] (-45:0.3) -- (-45:0.5);
  \draw[double] (135:0.3) -- (-135:0.3);
\end{scope}
\begin{scope}
  \draw (135:0.3) -- (135:0.5);
  \draw (-135:0.3) -- (-135:0.5);
  \draw (135:0.3)-- (45:0.3) -- (-45:0.3) -- (-135:0.3);
\end{scope}
\end{scope}}}\right)&=
    \ncol\left( \NB{\tikz[]{\begin{scope}
    \draw (0,-0.15) -- (0, 0.15); 
    \draw (135:0.5) .. controls  (135:0.25) and +(0,0) .. (0,0.15);
    \draw[double] (45:0.5) .. controls  (45:0.25) and +(0,0)
    .. (0,0.15) coordinate[pos=0.3] (v1);
    \draw (-135:0.5) .. controls  (-135:0.25) and +(0,0) .. (0,-0.15);
    \draw[double] (-45:0.5) .. controls  (-45:0.25) and +(0,0)
    .. (0,-0.15)
    coordinate[pos=0.3] (v2);
    \node at (v1) {$\myvertex$};
    \node at (v2) {$\myvertex$};
\end{scope}

}}\right) +
    \ncol\left( \NB{\tikz[]{\begin{scope}
  \draw[double] (45:0.5) .. controls (45:0.3) and
  (-45:0.3) .. (-45:0.5);
  \draw (135:0.5) .. controls (135:0.3) and (-135:0.3).. (-135:0.5);
\end{scope}
}}\right),\\
    \label{eq_square_2}
\ncol\left( \NB{\tikz[]{\begin{scope}[xscale=-1]
\begin{scope}
  \draw (45:0.3) -- (45:0.5);
  \draw (-45:0.3) -- (-45:0.5);
  \draw[double] (135:0.3) -- (-135:0.3);
\end{scope}
\begin{scope}
  \draw[double] (135:0.3) -- (135:0.5);
  \draw[double] (-135:0.3) -- (-135:0.5);
  \draw[double] (-45 :0.3) -- (-135:0.3) -- (135:0.3)-- (45:0.3);
  \draw (-45:0.3) -- (45:0.3);
\end{scope}
\end{scope}}}\right)&=
    \ncol\left( \NB{\tikz[]{}}\right) +
    \ncol\left( \NB{\tikz[]{\input{\imagesfolder/12v12}}}\right), \\
    \label{eq_square_3}
    \ncol\left( \NB{\tikz[]{\begin{scope}
\begin{scope}
  \draw (45:0.3) -- (45:0.5); 
  \draw (-45:0.3) -- (-45:0.5); 
  \draw [double](135:0.3) -- (-135:0.3); 
\end{scope}
\begin{scope}
  \draw (135:0.3) -- (135:0.5); 
  \draw (-135:0.3) -- (-135:0.5); 
  \draw (-45 :0.3) -- (-135:0.3); 
  \draw (135:0.3)-- (45:0.3); 
  \draw[double] (-45:0.3) -- (45:0.3); 
\end{scope}
\end{scope}}}\right)&=
    \ncol\left( \NB{\tikz[rotate=90]{\begin{scope}
  \draw (45:0.5) .. controls (45:0.3) and
  (-45:0.3) .. (-45:0.5);
  \draw (135:0.5) .. controls (135:0.3) and (-135:0.3).. (-135:0.5);
\end{scope}}}\right) +
    2\, \ncol\left( \NB{\tikz[]{}}\right), \\
    \label{eq_square_4}
    \ncol\left( \NB{\tikz[]{\input{\imagesfolder/Square2211}}}\right)&=
    \ncol\left( \NB{\tikz[rotate=90]{\begin{scope}[xscale=-1]
    \draw (0,-0.15) -- (0, 0.15); 
    \draw [double](135:0.5) .. controls  (135:0.25) and +(0,0) .. (0,0.15) coordinate[pos=0.3] (v1); 
    \draw (45:0.5) .. controls  (45:0.25) and +(0,0) .. (0,0.15); 
    \draw (-135:0.5) .. controls  (-135:0.25) and +(0,0) .. (0,-0.15);
    \draw[double] (-45:0.5) .. controls  (-45:0.25) and +(0,0)
    .. (0,-0.15)
    coordinate[pos=0.3] (v2);
    \node at (v1) {$\myvertex$};
    \node at (v2) {$\myvertex$};
\end{scope}}}\right) +
    \ncol\left( \NB{\tikz[]{\begin{scope}
    \draw (0,-0.15) -- (0, 0.15); 
    \draw [double](135:0.5) .. controls  (135:0.25) and +(0,0) .. (0,0.15) coordinate[pos=0.3] (v1); 
    \draw (45:0.5) .. controls  (45:0.25) and +(0,0) .. (0,0.15); 
    \draw (-135:0.5) .. controls  (-135:0.25) and +(0,0) .. (0,-0.15);
    \draw[double] (-45:0.5) .. controls  (-45:0.25) and +(0,0)
    .. (0,-0.15)
    coordinate[pos=0.3] (v2);
    \node at (v1) {$\myvertex$};
    \node at (v2) {$\myvertex$};
\end{scope}}}\right), \\
    \label{eq_sq5}
    \ncol\left( \NB{\tikz[]{\input{\imagesfolder/Square2222}}}\right)=
    2\, \left[ \ncol\left( \NB{\tikz[rotate=90]{\begin{scope}
  \draw[double] (45:0.5) .. controls (45:0.3) and
  (-45:0.3) .. (-45:0.5);
  \draw[double] (135:0.5) .. controls (135:0.3) and (-135:0.3).. (-135:0.5);
\end{scope}}}\right)\right. &+ \left.\ncol\left(\NB{\tikz[]{}}\right) + \ncol\left( \NB{\tikz[]{\input{\imagesfolder/4v_edges22bullets}}}\right)+
    \ncol\left( \NB{\tikz[rotate=90]{\input{\imagesfolder/4v_edges22bullets}}}\right) \right].
\end{align}
\end{proposition}

Note that the first term on the right-hand side of~\eqref{eq_square_2} can be further simplified via~\eqref{eq_square_1}. 

The relation \eqref{eq_loop_zero} says that a loop in a graph
evaluates it to $0$. Relations \eqref{eq_defects} allow to simplify dots on thick strands and to cancel and move around valency two vertices: two such dots on a single double strand can be removed, and a dot can be moved from one double strand to an adjacent strand. Here by a \emph{double strand} we mean a  consecutive chain of double edges separated by valency two vertices, with valency three vertices at its two endpoints.    

Four types of digon regions are possible in webs, and relations \eqref{eq_digon_1} and \eqref{eq_digon_2} show how to 
simplify a web with a digon region when computing its Tait number $t_4$.  

There are four types of triangular regions, and all can be simplified (collapsed) to a point, sometimes at the cost of scaling $t_4$ by two, see relations \eqref{eq_triangle_1}, \eqref{eq_triangle_2}. In particular, relation (\ref{eq_triangle_1}a) allows to replace each vertex of type 222 in a web for three vertices of type 112. We will make use of this in the proof of Proposition \ref{prop:t4computable}.

In a square region, any quadruple of thicknesses $1$ and $2$ of the four edges is possible, leading to six possibilities, up to rotation and reflection $1111, \ 1112, \ 1122, \ 1212, \ 1222, \ 2222$.
For five of these cases there are reduction formulas allowing to simplify the regions when computing $t_4$, see equations \eqref{eq_square_1}--\eqref{eq_sq5}. 
 The case 1111, when the four edges in the square all have thickness one, cannot be reduced, apparently.  

None of the regions with five and more sides seem to simplify in a way similar to the above proposition, as an $\N$-linear combination of diagrams with fewer regions or lower complexity.

\begin{proof}[Proof of Proposition \ref{prop:relations}]
Relations follow by case by case inspection. We include a proof of
relation (\ref{eq_square_1}). In Table~\ref{tab:checkcol}, the matrix
entry $(a,b)$ is the number of $4$-colorings of the web $b$ in the top
row whose endpoints have been assigned color combination $a$ in the
first column. Notice that any color must appear an even number of times on the boundary of a web without defects, which helps to list possible colorings. For the middle web, with two defects, one checks directly that only two of the four types of boundary colorings shown in the table may appear. In the table and in the notations throughout the paper, $\{i,j,k,l\} = \{1,2,3,4\}$. 
\begin{table}
\begin{tabular}{ |c|c|c|c| } 
 \hline
&  \,\NB{\tikz[]{\input{\imagesfolder/4v_square}}} \, &  \, \NB{\tikz[]{\input{\imagesfolder/4v_dumbell2bullet}}}\, & \,  \NB{\tikz[]{\input{\imagesfolder/4v_edges12}}} \, \\ \hline 
$\begin{array}{cc}
i & ij  \\
i & ij \end{array}$ & 1  & 0 & 1 \\ \hline
$ \begin{array}{cc}
k & ij  \\
k & ij \end{array}$& 2 & 1 & 1 \\ \hline
$\begin{array}{cc}
j & ij  \\
k & ik \end{array}$ & 0 & 0 & 0 \\ \hline
$\begin{array}{cc}
k & ij  \\
j & ik \end{array}$ & 1 & 1 & 0 \\ \hline
\end{tabular}
\\
\caption{Colorings illustrating the proof of relation \eqref{eq_square_1} in Proposition \ref{prop:relations}. Left column denotes coloring of the endpoints, with $\{i,j,k,l\} = \pigments$.}\label{tab:checkcol}
\end{table} \end{proof}

\begin{proposition}\label{prop_tutte}

The following local Tutte-like relations hold:
\begin{align} 
\label{eq_tutte_1}
\ncol\left(\NB{\tikz[]{\begin{scope}
    \draw[double] (0,-0.15) -- (0, 0.15); 
    \draw (135:0.5) .. controls  (135:0.25) and +(0,0) .. (0,0.15);
    \draw (45:0.5) .. controls  (45:0.25) and +(0,0) .. (0,0.15); 
    \draw (-135:0.5) .. controls  (-135:0.25) and +(0,0) .. (0,-0.15);
    \draw (-45:0.5) .. controls  (-45:0.25) and +(0,0) .. (0,-0.15);         
  \end{scope}}}\right) + \ncol\left(\NB{\tikz[rotate=90]{}} \right) &=
    \ncol\left(\NB{\tikz[rotate=90]{}} \right)+ \ncol\left(\NB{\tikz[]{}} \right), \\
    \label{eq_tutte_2}
    \ncol\left(\NB{\tikz[rotate=180]{\input{\imagesfolder/12v12}}}\right) + \ncol\left(\NB{\tikz[rotate=180]{}} \right) &=
    \ncol\left(\NB{\tikz[rotate=180]{}} \right)+ \ncol\left(\NB{\tikz[rotate=-90]{\input{\imagesfolder/11v22}}} \right), \\ 
    \label{eq_tutte_3}
     \ncol\left(\NB{\tikz[rotate=90]{\input{\imagesfolder/12defv21}}}\right) + \ncol\left(\NB{\tikz[rotate=270]{\input{\imagesfolder/12defv21}}} \right) &=
    \ncol\left(\NB{\tikz[xscale=-1]{\input{\imagesfolder/12defv21}}} \right)+ \ncol\left(\NB{\tikz[yscale=-1]{\input{\imagesfolder/12defv21}}} \right).\end{align}
\end{proposition}

\begin{proof}
Via case by case inspection in a similar way to that of previous proposition. 
\end{proof}

\begin{proposition}
    \label{prop:t4computable}
    Relations in Propositions~\ref{prop:relations} and \ref{prop_tutte} determine $t_4$.
\end{proposition}

This statement and its proof are analogous to \cite[Theorem XII \& XIII]{Tutte1947}.

\begin{proof}
Assume that the web $\Gamma$ is connected and does not contain
vertices of type $222$ (otherwise, use relation
(\ref{eq_triangle_1}a)), and denote by $k_\Gamma$ the number of
vertices of type 112 in $\Gamma$. It is clear that $k_{\Gamma}$ is
even, since it is twice the number of thick strands. Using relation
(\ref{eq_defects}a), we can further assume that each thick strands
contains at most one defect. 

The proof follows by induction on $k_\Gamma$. If $k_{\Gamma}$ is zero, then $t_4(\Gamma)$ is determined by relation \eqref{eq_add_circle} (and possibly by (\ref{eq_loop_zero}b)).

Assume now that the statement holds for any web $\Gamma'$ with $k_{\Gamma'} <n$, and let $k_\Gamma = n$.

Observe that \eqref{eq_quadruple_v} and  \eqref{eq_tutte_1} together with inductive hypothesis imply that the local changes
\begin{equation*}\label{eq:localmove}
 \NB{\tikz[scale=1.2]{\input{\imagesfolder/4v_dumbell-bullet}}}\longleftrightarrow
    \NB{\tikz[scale=1.2]{\input{\imagesfolder/4v_dumbell-bullet-intro}}}, \hspace{2cm} \NB{\tikz[scale=1.2]{}} \longleftrightarrow
\NB{\tikz[rotate=90, scale=1.2]{}}, 
\end{equation*}
preserve the computability of $t_4$ via the above relations.

Define $\Gamma^\times$ as the $4$-valent regular planar graph obtained from $\Gamma$ by contracting thick strands \begin{equation*}
   \NB{\tikz[scale=1.2]{}} \mapsto \NB{\tikz[]{}}\qquad \qquad \NB{\tikz[scale=1.2]{\input{\imagesfolder/4v_dumbell-bullet}}} \mapsto \NB{\tikz[]{\begin{scope}[scale=1.2]
  \draw (45:0.5) -- (225:0.5);
  \draw (135:0.5) -- (315:0.5);
  \node at (0,0) {$\circ$};
\end{scope}}}. 
\end{equation*}

The standard Euler characteristic argument for $\Gamma^\times$ implies that it contains one of the following local pieces (with or without labelling $\circ$ of the vertices):
\begin{equation}\label{options}
\NB{\tikz[]{\begin{scope}
\clip (225:0.6) rectangle (45:0.6);
\draw (0,0) -- (315:0.5);
\draw (0,0) -- (225:0.5);
\draw (0,0) .. controls (45:0.7) and (135:0.7) .. (0,0);
\end{scope}}},\qquad\qquad\NB{\tikz[]{}} \qquad\qquad \text{or}\qquad\qquad\NB{\tikz[]{}}
\end{equation}

 Observe that if two webs $\Gamma_1$ and $\Gamma_2$ lead to the same graph $\Gamma_1^\times = \Gamma_2^\times$, then either $t_4(\Gamma_1)$ and $t_4(\Gamma_2)$ are both computable from relations in the statement, or neither of them is. Then, for the first two cases in \eqref{options} we can assume that the web $\Gamma$
contains a local piece of the form
\begin{equation*}
\NB{\tikz[]{}}\qquad \qquad \NB{\tikz[]{}} ,
\end{equation*} and relations (\ref{eq_loop_zero}a) and (\ref{eq_digon_1}a) determine $t_4(\Gamma)$. 

For the last case, we can assume that $\Gamma$ contains a local piece of the form
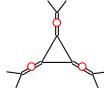
\begin{equation*}
\NB{\tikz[]{\input{\imagesfolder/4v_gamma}}}
\end{equation*}
where each red defect may or may not appear.  Using relations (\ref{eq_triangle_1}b), \eqref{eq_defects} and \eqref{eq_tutte_2} we get the following chain of local transformations
\begin{align*}
\ncol\left({\NB{\tikz[]{\input{\imagesfolder/4v_gamma}}}}\right) \quad & = \quad \ncol \left({\NB{\tikz[]{\begin{scope}[scale=1.2]
    \draw[double] (90:0) -- (90: 0.4); 
    \draw[double] (-150:0) -- (-150: 0.4);
    \draw[double] (-30:0) -- (-30: 0.4); 
    \draw (90:0.4) -- (100:0.55);
    \draw (90:0.4) -- (80:0.55);
    \draw (-150:0.4) -- (-140:0.55);
    \draw (-150:0.4) -- (-160:0.55);
    \draw (-30:0.4) -- (-20:0.55);
    \draw (-30:0.4) -- (-40:0.55);
    \node at (90:0.2) {$\myvertextwo$};
    \node at (-150:0.2) {$\myvertextwo$};
    \node at (-30:0.2) {$\myvertextwo$};
  \end{scope}}}}\right) \quad = \quad \ncol \left({\NB{\tikz[]{\input{\imagesfolder/4v_gamma2}}}}\right) \quad  \\ 
& = \quad \ncol\left({\NB{\tikz[]{\input{\imagesfolder/4v_gamma3}}}}\right) + \ncol\left({\NB{\tikz[]{\input{\imagesfolder/4v_gamma4}}}}\, \right) - \ncol\left({\NB{\tikz[]{\begin{scope}[scale=1.3, yscale=0.8]
 \draw (60:0.7) .. controls (60:0.3) and (-60:0.3).. (-60:0.7);
 \draw[double] (120:0.55) .. controls (120:0.3) and (-120:0.3).. (-120:0.55);
 \draw (120:0.55) -- (110:0.7);
\draw (120:0.55) -- (130:0.7);
 \draw (-120:0.55) -- (-110:0.7);
\draw (-120:0.55) -- (-130:0.7);
 \node at (180:0.18) {$\myvertextwo$};
\end{scope}

 }}}\, \right).
\end{align*}
A similar computation works if starting with an even number of dot defects shown in red. 
 Inductive hypothesis completes the proof. 
\end{proof}

The set of colorings of a given web can be endowed with an equivalence
relation induced by so-called Kempe moves. 

\begin{definition}
  Let $\Gamma$ be a web and $c$ and $c'$ be two colorings of $\Gamma$. We
  say that $c$ and $c'$ are related by a \emph{Kempe-move} if they agree on
  every edge except for edges on a closed even length cycle. Two colorings $c_1$ and $c_2$
  of $\Gamma$ are \emph{Kempe-equivalent} if they are related by a
  finite number of Kempe-moves. Classes for this equivalence relation
  are called \emph{Kempe-classes}.  
\end{definition}

We will make use of this definition in Sections~\ref{sec_state_sp} and ~\ref{sec_localization}.

\section{Foams and their evaluation}
\label{sec:foams-and-eval}
\begin{definition}\label{def_foam}
  A (closed) \emph{foam} $F$ is a finite 2-dimensional CW complex
  whose facets are labeled $1$ or $2$ (this label is the
  \emph{thickness} of that facet). Locally it is required to be
  homeomorphic to one of the local models depicted in
  Figure~\ref{fig:localmodelsfoams}, where grey color represents
  facets of thickness $2$.

  Points having neighborhood homeomorphic to an open disc are
  \emph{regular}. Points whose neighborhood is homeomorphic to the
  product of a tripod and an open interval are \emph{seam
    points}. Finally, those whose neighborhood is homeomorphic to the
  cone over the $1$-skeleton of a tetrahedron (see
  Figure~\ref{fig:cones}) are \emph{singular vertices}. The union of
  seams and singular vertices has a structure of a 4-valent graph
  denoted $\seam(F)$.
\end{definition}

\begin{figure}
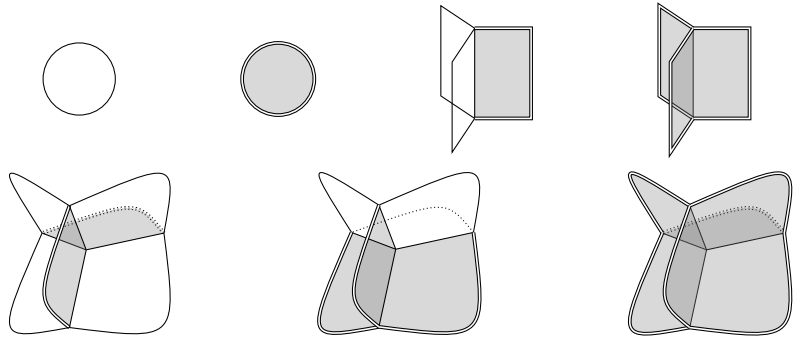

  \centering
\NB{\tikz[scale=0.6]{\input{\imagesfolder/4v_disk1}}}\qquad \qquad
\NB{\tikz[scale=0.6]{\input{\imagesfolder/4v_disk2}}}\qquad \qquad
\NB{\tikz[scale=1.5]{\input{\imagesfolder/4v_seam211}}}\qquad \qquad
\NB{\tikz[scale=1.5]{\input{\imagesfolder/4v_seam222}}} \\
\NB{\tikz[]{\input{\imagesfolder/4v_singvertex221111}}}\qquad \qquad
\NB{\tikz[]{\input{\imagesfolder/4v_singvertex222111}}}\qquad \qquad
\NB{\tikz[]{\input{\imagesfolder/4v_singvertex222222}}}
\caption{Standard neighborhoods of a regular point (first two diagrams in upper row), of a seam point of type $112$ and $222$ (third and fourth in upper row, respectively) and of a singular vertex of type $1^42^2$, $1^32^3$ and $2^6$ (first, second and third in second row, respectively).}
\label{fig:localmodelsfoams}
\end{figure}

\begin{figure}
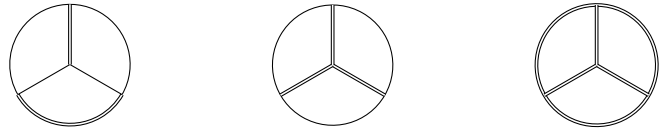

  \centering
  \NB{\tikz[]{\input{\imagesfolder/4v_cone221111}}}\qquad \qquad
  \NB{\tikz[]{\input{\imagesfolder/4v_cone222111}}}\qquad \qquad
  \NB{\tikz[]{\input{\imagesfolder/4v_cone222222}}}\qquad \qquad
  \caption{Standard neighbourhoods of singular vertices in
    Figure~\ref{fig:localmodelsfoams} above are homeomorphic to cones
    over the $K_4$ graph with $2,3$ and $6$ double edges,
    respectively.}\label{fig:cones}
\end{figure}

Foams are required to be PL-embedded in $\R^3$ and can be endowed with
additional combinatorial/geometrical data:
\begin{enumerate}
\item \emph{Defect lines}: A (possibly empty) set of PL-curves in the
  closure of facets of thickness 2 in general position with respect to
  seams, that we denote $\defect(F)$. Points in a defect line are not
  considered to be regular. A \emph{region} of a foam $F$ is a connected component of 
    \begin{equation*}
    \regions(F) \ := \ F\setminus (\seam(F)\cup \defect(F)).
    \end{equation*}
    Each region $r$ inherits  thickness $\thickness(r)\in\{1,2\}$ from the facet containing $r$. We call $\regions(F)$ \emph{the set of regions} of $F$. 
  
\item \emph{Decorations}: A (possibly empty) collection of dots which
  are on regular points. Dots floating in regions of thickness $1$ (resp.{} $2$) are labeled by symmetric polynomial in $1$ (resp.{} $2$)
  variable, with variables having degree $2$. For dots on a region of thickness $1$, we use the convention that an integer $a$ refers to the polynomial $X^a$ (we sometimes omit the label when $a=1$). For dots on facets of thickness $2$, we use $e_1$ to refer to the first elementary polynomial in two variables and Young diagrams refer to their corresponding Schur polynomial in two variables.

\item \emph{Kempe-specification}: A map from the set of singular vertices of type $2^6$ to $\{\kempesquare, \kempetriangle\}$. 
\end{enumerate}

A \emph{foam $F$ with boundary} is a CW-complex embedded in $\R^2\times [0,1]$, which have the same local models as a closed foam in $\R^2\times (0,1)$ and is homeomorphic to a product of a web with an interval in $\R^2\times [0, \epsilon)$ and in $\R^2\times (1-\epsilon, 1]$ for some positive $\epsilon$. The web $F\cap \R^2\times \{i\}$, denoted $\partial_iF$, is called the $i$-boundary of $F$, for $i \in \{0,1\}$. 

Webs and foams form a corbordism-like category denoted $\catFoam$, whose objects are webs and morphisms between $\Gamma_0$ and $\Gamma_1$ are foams $F$ with boundary such that $\Gamma_i = \partial_iF$ for $i\in\{0,1\}$, considered up to ambient isotopy relative to boundary and preserving all decorations. Identities are given by product of webs with $[0,1]$, and the composition comes from stacking foams and rescaling in the vertical direction. In this category, a closed foam gives a cobordism from the empty web $\emptyset_1$ to itself.

\subsection{Foam evaluation}\label{subsec_evaluation}

Throughout this section we'll work over the two-elements field
$\ftwo$ and use four variables $X_1, \dots, X_4$ which by convention
have degree $2$ each. The graded ring of symmetric
polynomials $\basering$ is denoted $\ourring$ and the ring $\ftwo[X_i,
\frac{1}{X_j+X_k}]_{\substack{1\leq i\leq 4\\ 1\leq j <k \leq 4}}$ is
denoted $\ourring'$. 

Schur polynomials (over $\ftwo$) in two variables $X$ and $Y$ 
forms a linear basis of $\ftwo[X,Y]^{S_2}$. They are parameterized by
Young diagrams $\lambda=(\lambda_1, \lambda_2)$ with at most two rows
$(\lambda_1\geq \lambda_2)$ and given by the following formula:
\[
  s_\lambda(X,Y) = \frac{X^{\lambda_1+1}Y^{\lambda_2}- X^{\lambda_2}Y^{\lambda_1+1}}{X-Y}.
\]

\begin{definition}
\label{def:colfoam}
    A \emph{coloring} of a foam $F$ is a map 
    \[
    c \ : \ \regions(F) \lra \mcP(\pigments)\]
    from the set of its regions to subsets of $\pigments=\{1,2,3,4\} $such that the cardinality of $c(r)$ is the thickness of 
     the region $r\in \regions(F)$, and $c$ satisfies the following requirements (we write $r_i$ for a region with $th(r_i)=i$, for $i=1,2$):
    \begin{enumerate}
        \item \label{it:211-foam} If three regions $r_1^a, r_1^b$ and $r_2$ are adjacent through a seam of type $112$, then  $c(r_2)=c(r_1^a) \sqcup c(r_1^b)$, that is, the coloring of $r_2$ is the disjoint union of those of $r_1^a$ and $r_1^b$. 
        \item \label{it:222-foam} If three regions $r_2^a, r_2^b$ and $r_2^d$ are adjacent through a seam of type $222$, then $c(r_2^a), c(r_2^b), c(r_2^d)$ are the three two-element subsets of a cardinality three subset of $\pigments$.

\item \label{it:defect-foam} If two regions $r_2^a$ and $r_2^b$ are adjacent through a defect, then $c(r_2^a)\cup c(r_2^b) = \pigments$, i.e., their  colors are complementary.
        \item \label{it:Kempe-foam} The colors of the regions around a $2^6$ singular vertex should follow a pattern according to the Kempe-specification of this singular vertex. These patterns are given on Figure~\ref{fig:KempePattern}. 
    \end{enumerate}
    \begin{figure}
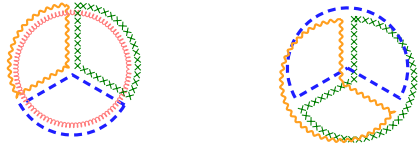

        \centering
        \NB{\tikz[]{\input{\imagesfolder/4v_conetriangle}}}\qquad \qquad \NB{\tikz[]{\input{\imagesfolder/4v_conesquare}}}        
        \caption{Relation between Kempe-specification and local
          coloring around a vertex of type $2^6$. On the left
          (resp. right) a coloring for a Kempe specification
          $\kempetriangle$ (resp. $\kempesquare$) is shown.
          These two types of colorings form the two Kempe-classes of
          the web \NB{\tikz[scale=0.3]{\input{\imagesfolder/4v_cone222222}}}.}
\label{fig:KempePattern}
    \end{figure}
    If $c$ is a coloring of $F$, the pair $(F,c)$ is called a \emph{colored} foam.
\end{definition}

\begin{remark}
    In Definition \ref{def:colfoam}, the constraints (\ref{it:211-foam}--\ref{it:defect-foam}) are inherited from colorings of web.  
\end{remark}

\begin{example}
 \begin{figure}
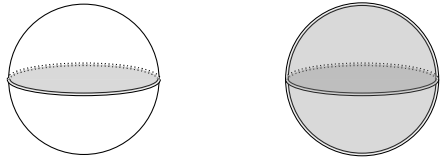

        \centering
        \NB{\tikz[]{\input{\imagesfolder/4v_sphere1}}}\qquad \qquad \NB{\tikz[]{\input{\imagesfolder/4v_sphere2}}}
        \caption{112 theta-foam and $222$ theta-foam, having 12 and 24 different colorings, respectively.}
\end{figure}

\begin{itemize}
\item Coloring of a foam $S\subset \R^3$ which is a connected closed thin, respectively thick, surface is a subset of $\mcP(\pigments)$ of cardinality one, respectively two. Thus, such a foam has four, respectively six, colorings. 
\item By a $112$ \emph{theta-foam} $\Theta_{112}$ we mean a foam which is a union of one thick and two thin disks along the circle. Its colorings are ordered pairs $(i,j)$ of pigments, with $i\not= j$, and this foam has 12 colorings.   
\item A $222$ \emph{theta-foam} $\Theta_{222}$ is given by gluing three thick disks along the boundary circles. Its colorings are in a bijection with ordered triples $(i,j,k)$ of distinct pigments. Each disk is colored by an unordered pair of two pigments out of these three, and the number of colorings is $24$. 
\end{itemize}
\end{example}

  \begin{question}
    Is there a meaningful characterization of colorable SL(4) foams?
  \end{question}
  
Let $(F,c)$ be a colored  closed foam and $i<j$ two distinct elements
of $\pigments$. The union of regions whose
color contains exactly one element of $\{i,j\}$ is a closed surface
denoted $F_{ij}=F_{ij}(c)$. Since $F$ is embedded in $\R^3$, the surface
$F_{ij}$ is orientable and its Euler characteristic is therefore
even. Hence we can define:
\begin{equation*}
\label{eq:QFc} 
Q(F,c)= \prod_{i<j \in \pigments} (X_i +X_j)^{\chi(F_{ij})/2}.
\end{equation*}
Recall that each decoration of $F$ belongs to a region $r$ and each region is assigned a
color $c(r)$ by the coloring $c$. Each decoration is labeled by a symmetric
polynomial in as many variables as the thickness of the regions it
floats in. Hence, we can evaluate it in the variable $\{X_i, i \in
c(r)\}$ and call this the \emph{colored evaluation} of the decoration.
For a colored foam $(F,c)$, define $P(F,c)$ to be the product of
colored evaluation of all decorations.

Finally define 
\begin{equation*}
  \eval[F,c] = \frac{P(F,c)}{Q(F,c)} \in \ourring' \qquad \text{and}
  \qquad \eval[F] = \sum_{c \in \colorings(F)} \eval[F,c]\in \ourring'.
\end{equation*}
The quantity $\eval[F]$ is called the \emph{evaluation} of the foam $F$.

The symmetric group on $4$ elements acts on the set of colorings of a
given foam and acts by permuting the variables
$(X_i)_{i \in \pigments}$. One can check that the following identity
holds:
\[
\eval[F, \sigma\cdot c]= \sigma\cdot \eval[F,c] 
\]
for any $\sigma$ in $S_4$.

\begin{proposition}
  Let $F$ be a foam $F\subset \R^3$ which admits a coloring.
  Suppose a thick facet $f$ of $F$  has a boundary
  circle $C$ with an odd number of vertices and there are no defects
  at the vertices along $C$. Then:
  \begin{itemize}
  \item The number of singular vertices of type $1^32^3$ is even.
  \item The number of singular vertices of type $1^42^2$ plus the
    number of singular vertices of type $2^6$ with $\kempetriangle$
    Kempe-specification is even.
  \end{itemize}
\end{proposition}

\begin{proof}
  Consider the regular neighborhood of $C$ intersected with the
  facets of $F$ which are not $f$ and which are adjacent to the seams
  included in $C$. It is an I-bundle $A$ over the circle, so in principle it could be either an annulus or a Möbius band. From a local coloring analysis,
  one obtains that this $I$-bundle is contained in
  $F_{1,2}$, so it has to be an annulus. Hence if we
  orient the circle we can speak of the right-hand and left-hand facets.

  The thicknesses of the regions conforming $A$ may be $1$ or $2$. At any
  given point of $C$ which is not a singular vertex, the thicknesses
  on the right-hand and left-hand facets is the same. This thickness
  changes along $C$ when encountering a singular vertex of type
  $1^32^3$. This proves the first statement.
  
  Regardless of their thicknesses, at each non-singular point of the
  circle, either the coloring of the right-hand facet or that of the
  left-hand facet (but not both) contains the pigment \colori. This
  arrangement changes exactly when encountering a vertex of type
  $1^42^2$ or $2^6$ with $\kempetriangle$ Kempe-specification.
  \end{proof}

\begin{definition}\label{def:degreeF}
  The degree $\deg(F)$ of a foam $F$ is an integer which is computed
  as follows: 
  \begin{align*}
    \deg(F) &= \sum_{\text{dots on $F$}} \deg(\text{polynomial
              labeling the dot}) -3 \sum_{\substack{\text{$f$ facet,}  \\ 
              \text{$\thickness(f) =1$}}} {\chi(f)}
-4 \sum_{\substack{\text{$f$ facet,}  \\ 
              \text{$\thickness(f) =2$}}} {\chi(f)}     \\ & + 5 \#  \{ \text{seam intervals of
    type $112$} \} \notag \\ & + 6 \# \{ \text{seam intervals of type
                         $222$}\} \notag
    \\ & - 6\# \{ \text{singular vertex of type $2^6$ and of type $1^32^3$}\} \notag
    \\ & - 5\# \{ \text{singular vertex of type $1^42^2$ with even number of
    defect lines going through}\} \notag
    \\ &- 6\# \{ \text{singular vertex of type $1^42^2$ with odd
number of defect lines going through}\}. \notag
  \end{align*}
\end{definition}

Recall that degrees are in variables that have degree $2$, so that the
first term in the definition of $\deg(F)$ is even. The above
definition is designed so that the following lemma holds, which gives
a convenient way to compute the degree of many foams.

\begin{lemma}\label{lem:degree}
Suppose that  a foam $F$ carries no dots 
and admits a  coloring $c$. Then:
\begin{align} \label{eq:degfromsurfaces}
\deg(F)= - \sum_{1\le i<j\le 4}\chi(F_{ij}(c)).
\end{align}
In particular, since $F$ is embedded in $\R^3$, if it admits a coloring then $\deg(F) \in 2\Z$. 

\end{lemma}
\begin{proof}
  The proof is an elementary but cumbersome inspection of different
  bicolored surfaces $F_{ij}$ of each local model for any (local)
  coloring. For instance, a singular seam interval of type $112$ whose adjacent
  regions are colored by, say,
  $\{1\}$, $\{2\}$ and $\{1,2\}$, appears in exactly
  $5$ bicolored surfaces: $F_{12}, F_{13}, F_{14}, F_{23}$ and $F_{24}$. This fits with the first coefficient $5$
  in the formula of Definition \ref{def:degreeF}.
\end{proof}

Using the same arguments as in \cite{RobWag} or \cite{KhovRob1},  one obtains the following result. 

\begin{proposition}
    If $F$ is a foam of degree $n$, then $\eval[F]$ is an element of $\ourring$ of degree $n$. In particular, if $\deg(F)<0$, then $\eval[F]=0$.
\end{proposition}

\begin{example}
  Closed surfaces (of thickness 1 or 2) with or without defect lines are a special case of foams, the simplest ones in a sense. Let us
  compute $\eval[F]$ for $F=\SS_1(\bullet^a))$, the evaluation of the
  sphere of thickness one, with no defect lines and $a$ $\bullet$ on
  it. In this case there is only one region and therefore $F$ has four colorings in bijection with elements of $\pigments$. Let us denote them
  $c_i, \, i \in \pigments$. One has:
  \begin{equation*}
    \eval[F, c_i] = \frac{X_i^a}{\displaystyle\prod_{j\neq i}(X_i+X_j)},
  \end{equation*}
  so that:
  \begin{equation*}
  \eval[F]=
    \displaystyle\sum_{i} \frac{X_i^a}{\displaystyle\prod_{j\neq i}(X_i+X_j)}
    =\frac{\displaystyle\sum_{i} X_i^a \displaystyle\prod_{j,k\neq i, j < k}(X_j+X_k)}{ \displaystyle\prod_{j< k} (X_j+X_k)} 
    = H_{a-3},
  \end{equation*}
  where $H_{d}$ is the complete symmetric polynomial 
  \begin{equation*} H_d=\displaystyle\sum_{t_1+t_2+t_3+t_4=d} X_1^{t_1} X_2^{t_2}X_3^{t_3}X_4^{t_4},
  \end{equation*}
  and $H_d=0$ if $d<0$.  
  \end{example}
  
  \begin{example}
  Consider $F=\SS_2(s_\lambda)$, the sphere of thickness 2, with no defect
  lines and a decoration by the Schur polynomial $s_\lambda$. As
  before, there is only one region and 
  $\left({\begin{smallmatrix} 4 \\ 2 \end{smallmatrix}}\right)= 6 $ colorings, in bijection with (unordered) pairs of elements of
  $\pigments$. Let $c_{ij}$ be one of such colorings, with $i<j$. One has:
  \begin{equation*}
    \eval[F, c_{ij}] = \frac{s_{\lambda}(X_i, X_j)}{(X_i+X_k)(X_i+X_\ell)(X_j+X_k)(X_j+X_\ell)}
  \end{equation*}
  with $\{i,j,k,\ell\} = \pigments$, so that:
  \begin{equation*}
    \eval[F] = s_{\lambda\setminus \rho(2,2)}
  \end{equation*}
  where $s_{\lambda\setminus \rho(2,2)}$ is the Schur polynomial
  associated with the Young diagram $\lambda$ with the two boxes of
  each row removed if possible (if $\lambda_i\ge 2\, \, \forall i$). If $\lambda_i=1$ for some $i$, we set $s_{\lambda\setminus \rho(2,2)}=0$.
\end{example}

\begin{example}
Let $F$ be the product of a web $\Gamma$ with the circle $\SS^1$ and no
decoration. Colorings of $\Gamma$ and $F$ are in obvious one-to-one correspondence. Let $c$ be a coloring of $F$. For any $i,j\in \pigments,$ $i<j$,
the surface $F_{ij}$ is a torus, so that:
\begin{equation*}
  \label{eq:1}
  \eval[F,c] = 1.
\end{equation*}
and since we work in characteristic $2$, we get
\begin{equation*}
  \label{eq:2}
  \eval[F,c] = \ncol(\Gamma) = \begin{cases}
      1 & \text{if} \ \Gamma =\emptyset_1, \\ 
      0 & \text{otherwise}.
  \end{cases}
\end{equation*}
\end{example}

The following relation, whose proof is by inspection, implies that Kempe-specifications of $2^6$-singular vertices can be described via foams without such specifications.

\begin{lemma}\label{lemma_rel_3}
  The following relations hold: 
  \begin{align}
    \eval[{\NB{\tikz[font= \tiny, scale=0.8]{\input{\imagesfolder/4v_singvertex222222triangle}}}}] &=
    \eval[{\NB{\tikz[font= \tiny, scale=0.8]{\input{\imagesfolder/4v_singvertex222222bubble}}}}] \\                      \eval[{\NB{\tikz[font= \tiny, scale=0.8]{\input{\imagesfolder/4v_singvertex222222square}}}}] &=
    \eval[{\NB{\tikz[font= \tiny, scale=0.8]{\input{\imagesfolder/4v_singvertex222222bubble_defect}}}}]
\end{align}
\end{lemma}

\subsection{Decomposition relations}
\label{sec:neck-cutt-relat} 
In this section skein relations on foams are determined. They lift $t_4$ web evaluation relations of Proposition~\ref{prop:relations} (save the relation~\eqref{eq_sq5}) to isomorphisms of state spaces, as will be shown in Theorem~\ref{thm_dir_sum_rels}.

\begin{lemma}\label{lemma_rel_0}
  The following relations hold:
  \begin{align}
  \label{eq_rel_0}
    \eval[{\NB{\tikz[scale=1.2]{\begin{scope}
  \draw (0,1.2) circle (0.4 and 0.2);
  \draw[densely dotted] (0.4,0) arc (0: 180: 0.4 and 0.2);
  \draw (0.4,0) arc (0: -180: 0.4 and 0.2); 
  \draw (0.4, 0) -- +(0,1.2);
  \draw (-0.4, 0) -- +(0,1.2);
\end{scope}
}}}] &= \sum_{a+b+d=3} E_{d}\cdot
    \eval[{\NB{\tikz[scale=1.2]{\input{\imagesfolder/4v_cupcupdotted}}}}]=\sum_{a=0}^3 \left( \sum_{b+d =3-a} E_{d}\cdot
    \eval[{\NB{\tikz[scale=1.2]{\input{\imagesfolder/4v_cupcupdotted}}}}]\right) \\
    \label{eq_rel_1}
   \eval[{\NB{\tikz[scale=1.2]{\input{\imagesfolder/4v_tube2}}}}] &=\sum_{\alpha\in \YD(2,2)} 
    \eval[{\NB{\tikz[scale=1.2]{\input{\imagesfolder/4v_cupcupdotted2}}}}] 
  \end{align}
  In the second relation $\alpha\mapsto\alpha^{\star}$ is the involution on the set $\YD(2,2)$ 
  of six Young diagrams with at most two rows and at most two columns determined by:
  \[
    \emptyset^\star = \NB{\tikz{\draw (0,0) rectangle (0.6,0.6); \draw
      (0.3,0) -- (0.3, 0.6); \draw (0, 0.3) -- (0.6, 0.3);}}\qquad\qquad
    \NB{\tikz{\draw (0,0) rectangle (0.3,0.3);}}^\star = \NB{\tikz{\draw (0,0) rectangle (0.3,0.3); \draw
      (0.3,0) rectangle (0.6, 0.3); \draw (0, 0.3) rectangle (0.3, 0.6);}}\qquad\qquad
      \NB{\tikz{\draw (0,0) rectangle (0.3,0.3); \draw (0.3,0) rectangle (0.6,0.3); }}^\star = \NB{\tikz{\draw (0,0) rectangle (0.3,0.3); \draw
      (0,0.3) rectangle (0.3, 0.6);}}.
  \]  
  The six terms on the RHS (right-hand side) of equation~\eqref{eq_rel_1} are mutually-orthogonal idempotents $e_{\alpha}$. In particular, evaluating a thick sphere that appears in the composition $e_{\alpha}e_{\beta}$ yields $\delta_{\alpha,\beta}$. 
  The four terms on the RHS of~\eqref{eq_rel_0} are mutually-orthogonal idempotents $e(a)$, $0\le a\le 3$, and the corresponding sums of thin sphere evaluations in $e(a)e(b)$ yields $\delta_{a,b}$.  
\end{lemma}

The proof of these statements (and of the other lemmas of this
section) relies on a case by case inspection of the possible colorings
of the foams involved and some sometimes cumbersome but simple
manipulation of symmetric polynomials. They are very similar to proofs
of analogous relations in \cite{RobWag} or \cite{KhovRob1}. We include
such a proof for relation~\eqref{eq_rel_2}.

\begin{lemma}\label{lemma_rel_1}
  The following relations hold (compare to Proposition \ref{prop:relations} \eqref{eq_digon_1} and \textnormal{(\ref{eq_digon_2}a)}):
  \begin{align}
  \label{eq_rel_2}
    \eval[{\NB{\tikz[font= \tiny, scale=0.8]{\input{\imagesfolder/4v_digon211I}}}}] &=
    \eval[{\NB{\tikz[font= \tiny, scale=0.8]{\input{\imagesfolder/4v_digon211cup1}}}}] +
    \eval[{\NB{\tikz[font= \tiny, scale=0.8]{\input{\imagesfolder/4v_digon211cup2}}}}] \\
    \label{eq_rel_3}
    \eval[{\NB{\tikz[font= \tiny, scale=0.8]{\begin{scope}
    \fill[fill = gray, fill opacity =0.3] (0,-2) -- (0.5, -2) .. controls +(0.3,0.3) and +(-0.3, 0.3) .. (1.5,-2) -- (2,-2) -- (2,0) -- (1.5,0) .. controls +(-0.3,-0.3) and +(0.3, -0.3) .. (0.5,0) -- (0,0)-- cycle;
    \fill[fill = gray, fill opacity =0.3] (0.5,-2) .. controls
    +(0.3,-0.3) and +(-0.3, -0.3) .. (1.5,-2) -- (1.5,0) .. controls
    +(-0.3,0.3) and +(0.3, 0.3) .. (0.5,0) -- cycle;
    

  \draw[double] (0,0) -- (0.5, 0);
  \draw[double] (1.5,0) -- (2, 0);
  \draw[double] (0.5,0) .. controls +(0.3,0.3) and +(-0.3, 0.3) .. (1.5,0);
  \draw[double] (0.5,0) .. controls +(0.3,-0.3) and +(-0.3, -0.3) .. (1.5,0);
  \draw[double] (0,-2) -- (0.5, -2);
  \draw[double] (1.5,-2) -- (2, -2);
  \draw[double, densely dotted] (0.5,-2) .. controls +(0.3,0.3) and +(-0.3, 0.3) .. (1.5,-2);
  \draw[double] (0.5,-2) .. controls +(0.3,-0.3) and +(-0.3, -0.3) .. (1.5,-2);
  \draw (2,0) -- +(0, -2);
  \draw (0,0) -- +(0, -2);
  \draw (1.5,0) -- +(0, -2);
  \draw (0.5,0) -- +(0, -2);
\end{scope}}}}] &=
    \eval[{\NB{\tikz[font= \tiny, scale=0.8]{\input{\imagesfolder/4v_digon222cup1}}}}] +
    \eval[{\NB{\tikz[font= \tiny, scale=0.8]{\input{\imagesfolder/4v_digon222cup2}}}}] \\                          \label{eq_rel_4}
    \eval[{\NB{\tikz[font= \tiny, scale=0.8]{\input{\imagesfolder/4v_digon121I}}}}] &= \sum_{a+b=2}
    \eval[{\NB{\tikz[font= \tiny, scale=0.8]{\input{\imagesfolder/4v_digon121cup1}}}}]             
  \end{align}
In these relations, all dots are on facets of thickness $1$ (horizontal facets in the last equation are thin). The three terms on the RHS of~\eqref{eq_rel_4} are mutually-orthogonal idempotents $e_a$, $0\le a\le 2$ (with $b=2-a$ in the formula), and the corresponding decorated bubbles with one thin and one thick facet floating in a thin region in $e_a e_{a'}$ reduce to  $\delta_{a,a'}$ times the thin region.  
\end{lemma}

\begin{proof}[Proof of relation~\eqref{eq_rel_2}]
  Denote $F, G_1$ and $G_2$ the three foams involved in
  \eqref{eq_rel_2} read from left to right. Observe that foams $G_1$ and $G_2$ are identical except for their decoration distributions. Denote
  $H$ the foam $G_1$ (or $G_2$) with the dot removed.
  
  Every coloring $c$ of $F$ induces a coloring of $H$. Colorings of $H$
  which do not induce colorings of $F$ are precisely those for which
  the digon on the top and on the bottom are colored in a symmetric
  way. Meaning in particular that for those colorings, the dot of
  $G_1$ and the dot of $G_2$ are on facets with the same color. Since
  the dot distribution is the only difference between $G_1$ and $G_2$,
  for such a coloring $c'$, we have: $\eval[G_1,c'] = \eval[G_2, c']$,
  so that 
  \begin{equation}\label{eq:nonmatching}
    \eval[G_1,c'] +\eval[G_2,c'] =0.
  \end{equation}

  Suppose now that $c$ is a coloring of $F$, denote $c$ the induced
  coloring for $H$. It induces the same coloring of the top and bottom
  digons which completely characterizes the local behaviour of $c$ for
  $F$ and $H$. To fix notation say that this digon coloring is as
  follows:
  \[
    \NB
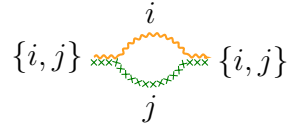
{\tikz[scale=1.5]{\begin{scope}
  \draw[expand style=\stylecoll] (-0.5,0) -- (-0.3,0) .. controls
  +(-60:0.3) and + (-120:0.3) .. node[below, pos=0.5,black] {$j$}
  (0.3,0)--(0.5,0) node[right,black] {$\{i,j\}$};
  \draw[yshift=0.05cm,expand style=\stylecoli] (-0.5,0) --  (-0.3,0) node[pos=0,
  left,black] {$\{i,j\}$}.. controls
  +(60:0.3) and + (120:0.3) .. (0.3,0) node[above, pos=0.5,black] {$i$}-- (0.5,0);
\end{scope}
}}
\]
  The only bicolored surface which is substantially different in
  $(F,c)$ and in $(H,c)$ is $\{i,j\}$-colored and $\chi(H_{ij}(c))=\chi(F_{ij}(c))+2$, so that:
  \begin{equation*}
    Q(G_1,c)=Q(G_2,c)=Q(H,c)=(X_i+X_j)Q(F,c).
  \end{equation*}
  The extra dots on $G_1$ and $G_2$ gives:
  \begin{equation*}
    P(G_1,c)= X_iP(H,c)=X_iP(F,c) \qquad .    P(G_2,c)= X_jP(H,c)=X_jP(F,c).
  \end{equation*}
  This implies:
  \begin{equation}\label{eq:matching}
    \eval[G_1,c] +\eval[G_2,c] = \frac{X_i}{X_i+X_j}\eval[F,c] +
    \frac{X_j}{X_i+X_j}\eval[F,c] = \eval[F,c].
  \end{equation}
  
  Summing over all colorings using  \eqref{eq:nonmatching} and
    \eqref{eq:matching} one obtains:
  \begin{equation*}
    \label{eq:5}
    \eval[F] = \eval[G_1]+ \eval[G_2].
  \end{equation*}
\end{proof}

\begin{lemma}\label{lemma_rel_4}
  The following relations hold (compare to Proposition \ref{prop:relations} \eqref{eq_defects}): 
  \begin{align}
  \label{eq_rel_A1}
    \eval[{\NB{\tikz[font= \tiny, scale=0.8]{\begin{scope}
  \filldraw[fill opacity=0.3, fill =gray] (0,0) rectangle (1,1);
  \draw[double] (0,0) -- +(1,0) coordinate[pos=0.3] (a0) coordinate[pos=0.7] (a1);
  \draw[double] (0,1) -- +(1,0) coordinate[pos=0.3] (a2)
  coordinate[pos=0.7] (a3);
  \draw[orange, thick] (a0) -- (a2);
  \draw[orange, thick] (a1) -- (a3);
  \foreach \x in {0,1,2,3}{
    \node at (a\x) {\tikz{\filldraw[fill=white, very thin] (0,0) circle
      (0.3mm);}};
  }
\end{scope}}}}] &=
    \eval[{\NB{\tikz[font= \tiny, scale=0.8]{\begin{scope}
  \filldraw[fill opacity=0.3, fill =gray] (0,0) rectangle (1,1);
  \draw[double] (0,0) -- +(1,0) coordinate[pos=0.3] (a0) coordinate[pos=0.7] (a1);
  \draw[double] (0,1) -- +(1,0) coordinate[pos=0.3] (a2)
  coordinate[pos=0.7] (a3);
  \draw[orange, thick] (a0) .. controls +(0,0.3) and +(0,0.3) .. (a1);
  \draw[orange, thick] (a2) .. controls +(0,-0.3) and +(0,-0.3) ..  (a3);
  \foreach \x in {0,1,2,3}{
    \node at (a\x) {\tikz{\filldraw[fill=white, very thin] (0,0) circle
      (0.3mm);}};
  }
\end{scope}}}}] \\ 
    \label{eq_rel_A2}
    \eval[{\NB{\tikz[font= \tiny, scale=0.8]{\begin{scope}
  \coordinate (yB) at (0,-1,0);
  \coordinate (yA) at (0,1,0);

  \coordinate (yb) at (0,-0.5,0);
  \coordinate (ya) at (0,0.5,0);

  \foreach \x in {120, 0, 240}{
\begin{scope}[zxplane=1, rotate=42]
    \coordinate (a\x) at (\x:0.5);
    \coordinate (A\x) at (\x:1);
  \end{scope}
  \begin{scope}[zxplane=-1, rotate =42]
    \coordinate (b\x) at (\x:0.5);
    \coordinate (B\x) at (\x:1);
  \end{scope}
\fill[opacity=0.3, gray] (yA) -- (A\x) -- (B\x) -- (yB);
\draw (A\x) -- (B\x);
\draw[double] (A\x) -- (yA);
}

\draw[double,densely dotted] (yB)--(B120) coordinate[pos=0.58] (b120);
\draw[double] (yB)--(B240);
\draw[double] (yB) -- (B0);
\draw (yA)--(yB);
\draw[double] (b120)--(B120);
\foreach \x in {0}{
\draw[thick] (a\x)--(b\x);
\draw[orange, thick] (a\x)--(b\x);
\node at (a\x) {\tikz{\filldraw[fill=white, very thin] (0,0) circle (0.3mm); }};
\node at (b\x) {\tikz{\filldraw[fill=white, very thin] (0,0) circle (0.3mm); }};
}

\end{scope}}}}] &=
    \eval[{\NB{\tikz[font= \tiny, scale=0.8]{\begin{scope}
  \coordinate (yB) at (0,-1,0);
  \coordinate (yA) at (0,1,0);

  \coordinate (yb) at (0,-0.5,0);
  \coordinate (ya) at (0,0.5,0);

  \foreach \x in {120, 0, 240}{
\begin{scope}[zxplane=1, rotate=42]
    \coordinate (a\x) at (\x:0.5);
    \coordinate (A\x) at (\x:1);
  \end{scope}
  \begin{scope}[zxplane=-1, rotate =42]
    \coordinate (b\x) at (\x:0.5);
    \coordinate (B\x) at (\x:1);
  \end{scope}
\fill[opacity=0.3, gray] (yA) -- (A\x) -- (B\x) -- (yB);
\draw (A\x) -- (B\x);
\draw[double] (A\x) -- (yA);
}

\draw[double,densely dotted] (yB)--(B120) coordinate[pos=0.58] (b120);
\draw[double] (yB)--(B240);
\draw[double] (yB) -- (B0);
\draw (yA)--(yB);
\draw[double] (b120)--(B120);
\foreach \x in {0}{
\draw[orange, thick] (a\x) .. controls +(0,-0.3, 0) and +(0,0,0).. (ya)
.. controls +(-0.5, 0, 0) and +(-0.5, 0,0) .. (yb) .. controls
+(0,0,0) and
+(0,+0.3, 0) .. (b\x);
\node at (a\x) {\tikz{\filldraw[fill=white, very thin] (0,0) circle (0.3mm); }};
\node at (b\x) {\tikz{\filldraw[fill=white, very thin] (0,0) circle (0.3mm); }};
}

\end{scope}
}}}]
\end{align}
\end{lemma}

\begin{lemma}\label{lemma_rel_6}
  The following relation holds (compare to Proposition \ref{prop:relations} \eqref{eq_quadruple_v}):
  \begin{align}
  \label{eq_rel_13}
    \eval[{\NB{\tikz[font= \tiny, scale=1]{\begin{scope}
  \foreach \x in {45, 135, 225, 315}{
\begin{scope}[zxplane=1, rotate=35]
    \coordinate (A\x) at (\x:1);
  \end{scope}
  \begin{scope}[zxplane=-1, rotate =35]
    \coordinate (B\x) at (\x:1);
  \end{scope}
\draw (B\x) -- (A\x);
}
\coordinate (A) at (0,1,0);
\coordinate (B) at (0,-1,0);
\begin{scope}[zxplane=1, rotate=35]
\coordinate (a90) at (90: 0.4);
\coordinate (a270) at (270: 0.4);  
\end{scope}
\begin{scope}[zxplane=-1, rotate=35]
\coordinate (b90) at (90:0.4);
\coordinate (b270) at (270: 0.4);
\end{scope}
\filldraw[fill opacity=0.3, fill= gray] (a90) --(a270) -- (b270) --
(b90) -- cycle;
\draw[double] (a90) -- (a270);
\draw[double] (b90) -- (b270);
\draw (A45) -- (a90) --(A135);
\draw (A315) -- (a270) --(A225);
\draw (B45) -- (b90) --(B135);
\draw (B315) -- (b270) --(B225);
\draw[thick, orange] (A) -- (B);
\node at (A) {$\myvertex$};
\node at (B) {$\myvertex$};
\end{scope}}}}] &= \eval[{\NB{\tikz[font= \tiny, scale=1]{\begin{scope}
  \foreach \x in {45, 135, 225, 315}{
\begin{scope}[zxplane=1, rotate=35]
    \coordinate (A\x) at (\x:1);
  \end{scope}
  \begin{scope}[zxplane=-1, rotate =35]
    \coordinate (B\x) at (\x:1);
  \end{scope}
  \begin{scope}[zxplane=0, rotate =35]
    \coordinate (C\x) at (\x:1);
  \end{scope}
\draw (B\x) -- (A\x);
}
\begin{scope}[zxplane=0, rotate=35]
\coordinate (c0) at (0: 0.4);
\coordinate (c180) at (180: 0.4);  
\end{scope}
\coordinate (A) at (0,1,0);
\coordinate (B) at (0,-1,0);
\coordinate (a) at (0,0.5,0);
\coordinate (b) at (0,-0.5,0);
\begin{scope}[zxplane=1, rotate=35]
\coordinate (a90) at (90: 0.4);
\coordinate (a270) at (270: 0.4);  
\end{scope}
\begin{scope}[zxplane=-1, rotate=35]
\coordinate (b90) at (90:0.4);
\coordinate (b270) at (270: 0.4);
\end{scope}
\filldraw[fill opacity=0.3, fill= gray] (a90) --(a270) .. controls
+(0,-0.3,0) and +(0,0,0) .. (a) .. controls +(0,0,0) and +(0,0.3,0)
.. (c0) .. controls +(0, -0.3, 0) and +(0,0,0) .. (b).. controls
+(0,0,0) and +(0,0.3,0) .. (b270) --
(b90)   .. controls
+(0,0.3,0) and +(0,0,0) .. (b).. controls +(0,0,0) and +(0,-0.3,0)
.. (c180) .. controls +(0, 0.3, 0) and +(0,0,0) .. (a).. controls
+(0,0,0)  and +(0,-0.3,0) .. cycle;
\draw[double] (a90) -- (a270);
\draw[double] (b90) -- (b270);
\draw (A45) -- (a90) --(A135);
\draw (A315) -- (a270) --(A225);

\draw (B45) -- (b90) --(B135);
\draw (B315) -- (b270) --(B225);
\draw[thick, orange] (A) -- (B);

\node at (A) {$\myvertex$};
\node at (B) {$\myvertex$};
\end{scope}
}}}] 
  \end{align}  
\end{lemma}

\begin{lemma}\label{lemma_rel_2}
  The following relations hold (compare to Proposition \ref{prop:relations} \eqref{eq_triangle_1}--\eqref{eq_triangle_2}): 
  \begin{align}
  \label{eq_rel_5}
    \eval[{\NB{\tikz[font= \tiny, scale=0.8]{\begin{scope}
  \foreach \x in {120, 0, 240}{
\begin{scope}[zxplane=1, rotate=42]
    \coordinate (a\x) at (\x:0.5);
    \coordinate (A\x) at (\x:1);
  \end{scope}
  \begin{scope}[zxplane=-1, rotate =42]
    \coordinate (b\x) at (\x:0.5);
    \coordinate (B\x) at (\x:1);
  \end{scope}
\fill[opacity=0.3, gray] (a\x) -- (A\x) -- (B\x) -- (b\x);
\draw (a\x) -- (b\x);
\draw (A\x) -- (B\x);
\draw[double] (A\x) -- (a\x);
\draw[double] (B\x) -- (b\x);
}
\draw (a0)--(a120)--(a240)--cycle;
\draw[densely dotted] (b0)--(b120);
\draw[densely dotted] (b120)--(b240);
\draw (b240) -- (b0);
\end{scope}}}}] &=
    \eval[{\NB{\tikz[font= \tiny, scale=0.8]{ 
\begin{scope}
  \coordinate (y1) at (0,0.3,0);
  \coordinate (y0) at (0,-0.3,0);
  \foreach \x in {120, 0, 240}{
\begin{scope}[zxplane=1, rotate=42]
    \coordinate (a\x) at (\x:0.5);
    \coordinate (A\x) at (\x:1);
  \end{scope}
  \begin{scope}[zxplane=-1, rotate =42]
    \coordinate (b\x) at (\x:0.5);
    \coordinate (B\x) at (\x:1);
  \end{scope}
\fill[opacity=0.3, gray] (a\x) -- (A\x) -- (B\x) -- (b\x) -- (y0) --(y1);
\draw (a\x) -- (y1) --(y0) -- (b\x);
\draw (A\x) -- (B\x);
\draw[double] (A\x) -- (a\x);
\draw[double] (B\x) -- (b\x);
}
\draw (a0)--(a120)--(a240)--cycle;
\draw[densely dotted] (b0)--(b120);
\draw[densely dotted] (b120)--(b240);
\draw (b240) -- (b0);
\end{scope}}}}] \\
    \label{eq_rel_6}
    \eval[{\NB{\tikz[font= \tiny, scale=0.8]{\begin{scope}
  \foreach \x in {120, 0, 240}{
\begin{scope}[zxplane=1, rotate=42]
    \coordinate (a\x) at (\x:0.5);
    \coordinate (A\x) at (\x:1);
  \end{scope}
  \begin{scope}[zxplane=-1, rotate =42]
    \coordinate (b\x) at (\x:0.5);
    \coordinate (B\x) at (\x:1);
  \end{scope}}
\fill[opacity=0.3, gray] (a120) -- (A120) -- (B120) -- (b120);
\fill[opacity=0.3, gray] (a240) -- (a0) -- (b0) -- (b240);
\draw (a0) -- (b0); 
\draw (a120) -- (b120);
\draw (a240) -- (b240);
\draw (A0) -- (B0); 
\draw (A120) -- (B120);
\draw (A240) -- (B240);
\draw (A0) -- (a0);
\draw[double] (A120) -- (a120);
\draw (A240) -- (a240);
\draw (B0) -- (b0);
\draw[double] (B120) -- (b120);
\draw (B240) -- (b240); 
\draw (a0)--(a120)--(a240);
\draw[double] (a0) -- (a240);
\draw[densely dotted] (b0)--(b120);
\draw[densely dotted] (b120)--(b240);
\draw[double] (b0) -- (b240);
\end{scope}}}}] &=
    \eval[{\NB{\tikz[font= \tiny, scale=0.8]{\begin{scope}
  \coordinate (y1) at (0,0.3,0);
  \coordinate (y0) at (0,-0.3,0);
  \foreach \x in {120, 0, 240}{
\begin{scope}[zxplane=1, rotate=42]
    \coordinate (a\x) at (\x:0.5);
    \coordinate (A\x) at (\x:1);
  \end{scope}
  \begin{scope}[zxplane=-1, rotate =42]
    \coordinate (b\x) at (\x:0.5);
    \coordinate (B\x) at (\x:1);
  \end{scope} 
\draw (a\x) -- (y1) --(y0) -- (b\x);
\draw (A\x) -- (B\x);
}
\fill[opacity=0.3, gray] (a120) -- (A120) -- (B120) -- (b120) -- (y0) -- (y1);
\fill[opacity=0.3, gray] (a240) -- (a0) -- (y1);
\fill[opacity=0.3, gray] (b240) -- (b0) -- (y0);
\draw (a0)--(a120)--(a240);
\draw [double] (a240)--(a0);
\draw[double] (A120) -- (a120);
\draw (A0) -- (a0);
\draw (A240) -- (a240);
\draw [double] (b240)--(b0);
\draw[double] (B120) -- (b120);
\draw (B0) -- (b0);
\draw (B240) -- (b240);
\draw[densely dotted] (b0)--(b120);
\draw[densely dotted] (b120)--(b240);
\end{scope}

}}}] \\
    \label{eq_rel_7}
    \eval[{\NB{\tikz[font= \tiny, scale=0.8]{\input{\imagesfolder/4v_triangle222222I}}}}] &=
    \eval[{\NB{\tikz[font= \tiny, scale=0.8]{\begin{scope}
  \coordinate (y1) at (0,0.3,0);
  \coordinate (y0) at (0,-0.3,0);
  \foreach \x in {120, 0, 240}{
\begin{scope}[zxplane=1, rotate=42]
    \coordinate (a\x) at (\x:0.5);
    \coordinate (A\x) at (\x:1);
  \end{scope}
  \begin{scope}[zxplane=-1, rotate =42]
    \coordinate (b\x) at (\x:0.5);
    \coordinate (B\x) at (\x:1);
  \end{scope}
  \fill[opacity=0.3, gray] (a\x) -- (A\x) -- (B\x) -- (b\x) -- (y0) --(y1);
\draw (a\x) -- (y1) --(y0) -- (b\x);
\draw (A\x) -- (B\x);
\draw[double] (A\x) -- (a\x);
\draw[double] (B\x) -- (b\x);
}
\fill[opacity=0.3, gray] (a120) --(a240) -- (y1);
\fill[opacity=0.3, gray] (a120) --(a0) -- (y1);
\fill[opacity=0.3, gray] (a0) --(a240) -- (y1);
\fill[opacity=0.3, gray] (b120) --(b240) -- (y0);
\fill[opacity=0.3, gray] (b120) --(b0) -- (y0);
\fill[opacity=0.3, gray] (b0) --(b240) -- (y0);
\draw[double] (a0)--(a120);
\draw[double] (a240)--(a120);
\draw[double] (a0)--(a240);
\draw[double,densely dotted] (b0)--(b120);
\draw[double,densely dotted] (b120)--(b240);
\draw[double] (b240) -- (b0);
\node[font=\tiny, left] at (y0) {$\kempetriangle$};
\node[font=\tiny, left] at (y1) {$\kempetriangle$};
\end{scope}}}}] +
     \eval[{\NB{\tikz[font= \tiny, scale=0.8]{\begin{scope}
  \coordinate (y1) at (0,0.3,0);
  \coordinate (y0) at (0,-0.3,0);
  \foreach \x in {120, 0, 240}{
\begin{scope}[zxplane=1, rotate=42]
    \coordinate (a\x) at (\x:0.5);
    \coordinate (A\x) at (\x:1);
  \end{scope}
  \begin{scope}[zxplane=-1, rotate =42]
    \coordinate (b\x) at (\x:0.5);
    \coordinate (B\x) at (\x:1);
  \end{scope}
  \fill[opacity=0.3, gray] (a\x) -- (A\x) -- (B\x) -- (b\x) -- (y0) --(y1);
\draw (a\x) -- (y1) --(y0) -- (b\x);
\draw (A\x) -- (B\x);
\draw[double] (A\x) -- (a\x);
\draw[double] (B\x) -- (b\x);
}
\fill[opacity=0.3, gray] (a120) --(a240) -- (y1);
\fill[opacity=0.3, gray] (a120) --(a0) -- (y1);
\fill[opacity=0.3, gray] (a0) --(a240) -- (y1);
\fill[opacity=0.3, gray] (b120) --(b240) -- (y0);
\fill[opacity=0.3, gray] (b120) --(b0) -- (y0);
\fill[opacity=0.3, gray] (b0) --(b240) -- (y0);
\draw[double] (a0)--(a120);
\draw[double] (a240)--(a120);
\draw[double] (a0)--(a240);
\draw[double,densely dotted] (b0)--(b120);
\draw[double,densely dotted] (b120)--(b240);
\draw[double] (b240) -- (b0);
\node[font=\tiny, left] at (y0) {$\kempesquare$};
\node[font=\tiny, left] at (y1) {$\kempesquare$};
\end{scope}}}}]\\
     \label{eq_rel_8}
    \eval[{\NB{\tikz[font= \tiny, scale=0.8]{\input{\imagesfolder/4v_triangle7BI}}}}] &=
    \eval[{\NB{\tikz[ scale=0.8]{\input{\imagesfolder/4v_triangle7Bcup1}}}}] +
     \eval[{\NB{\tikz[ scale=0.8]{\input{\imagesfolder/4v_triangle7Bcup2}}}}]   
  \end{align}
\end{lemma}

\begin{lemma}\label{lemma_rel_5}
  The following relations hold (compare to Proposition \ref{prop:relations} \eqref{eq_square_1}--\eqref{eq_square_4}): 
  \begin{align}
\label{eq_rel_10}
    \eval[{\NB{\tikz[font= \tiny, scale=1]{\input{\imagesfolder/4v_square1112I}}}}] &=
    \eval[{\NB{\tikz[font= \tiny, scale=1]{\begin{scope}
  \foreach \x in {45, 135, 225, 315}{
\begin{scope}[zxplane=1, rotate=35]
    \coordinate (a\x) at (\x:0.5);
    \coordinate (A\x) at (\x:1);
  \end{scope}
  \begin{scope}[zxplane=-1, rotate =35]
    \coordinate (b\x) at (\x:0.5);
    \coordinate (B\x) at (\x:1);
  \end{scope}
\draw (B\x) -- (A\x);
}

\begin{scope}[zxplane=0.3, rotate =35]
  \coordinate (M1) at (0.354, 0);
  \coordinate (M2) at (-0.354, 0);
  \draw (M1) -- (M2);
\end{scope} 

\begin{scope}[zxplane=-0.3, rotate =35]
  \coordinate (m1) at (0.354, 0);
  \coordinate (m2) at (-0.354, 0);

\end{scope} 

\begin{scope}
  \draw (b45) -- (m1);
  \draw (b315) -- (m1);
  \draw (m2) -- (M2);
  \draw (b225) -- (m2);
  \draw (b135) -- (m2);

\end{scope}

 \draw[double, densely dotted] (B135) -- (b135);

 \fill[opacity=0.3, gray] (m2)--(M2)--(a135) --(A135) -- (B135) -- (b135);
\draw[orange, thick]  ($0.3*(a45) +0.3*(B45)$) .. controls +(0, 0.2) and
+(0,0) ..   (M1) .. controls ($0.9*(a315)$) .. (M2)
.. controls +(0,0)  and +(0,0.2) .. ($0.4*(A135) +0.4*(b135)$)
.. controls +(0, -0.2) and +(0,0) .. (m2).. controls ($0.7*(b315)$) ..(m1)
.. controls +(0, 0) and +(0, -0.2) .. ($0.3*(a45) +0.3*(B45)$);
    \draw (m1) -- (m2);
  \draw (m1) -- (M1);
  \draw (a45) -- (M1);
  \draw (a315) -- (M1);
  \draw (a225) -- (M2);
  \draw (a135) -- (M2);

 \fill[opacity=0.3, gray] (m1)--(M1)--(a45) --(A45) -- (B45) -- (b45);
 \fill[opacity=0.3, gray] (a225) --(a315) -- (M1) -- (M2);
 \fill[opacity=0.3, gray] (m2) -- (m1) -- (b315) -- (b225);
 \draw (a45) -- (a135);
 \draw (a135) -- (a225);
 \draw[double] (a225) -- (a315);
 \draw (a315) -- (a45);
 \draw[densely dotted] (b45) -- (b135);
 \draw[densely dotted] (b135) -- (b225);
 \draw[double, densely dotted] (b225) -- (b315);
 \draw (b315) -- (b45);
 \draw[double] (A45) -- (a45);
 \draw[double] (A135) -- (a135);
\draw (A225) -- (a225);
\draw (A315) -- (a315);
 \draw[double] (B45) -- (b45);

\draw (B225) -- (b225);
\draw (B315) -- (b315);
\end{scope}}}}]+
    \eval[{\NB{\tikz[font= \tiny, scale=1]{\begin{scope}[rotate=0]
  \foreach \x in {45, 135, 225, 315}{
\begin{scope}[zxplane=1, rotate=35]
    \coordinate (a\x) at (\x:0.5);
    \coordinate (A\x) at (\x:1);
  \end{scope}
  \begin{scope}[zxplane=-1, rotate =35]
    \coordinate (b\x) at (\x:0.5);
    \coordinate (B\x) at (\x:1);
  \end{scope}

  \begin{scope}[zxplane=0, rotate =35]
    \coordinate (c\x) at (\x:0.5);
    \coordinate (C\x) at (\x:1);
  \end{scope}

\draw (B\x) -- (A\x);
}
\draw[very thin, dashed] (C135) .. controls ($0.45*(c135)$) and ($0.45*(c45)$) .. (C45);

\draw[very thin, dashed] (C225) .. controls ($0.45*(c225)$) and ($0.45*(c315)$) .. (C315);

 \draw[double, densely dotted] (B135) -- (b135);
 \fill[opacity=0.3, gray] (a135) --(A135) -- (B135) -- (b135);
 \fill[opacity=0.3, gray] (a45) --(A45) -- (B45) -- (b45);
 
 \filldraw[fill opacity =0.3, fill= gray] (a225) .. controls ($0.7*(a225) + 0.3*(b225)$) and
 ($0.7*(a315) + 0.3*(b315)$) .. (a315) coordinate[pos=0.65] (a270);

  \filldraw[fill opacity =0.3, fill= gray] (b225) .. controls ($0.3*(a225) + 0.7*(b225)$) and
 ($0.3*(a315) + 0.7*(b315)$) .. (b315) coordinate[pos=0.38] (b270);

  \fill[fill opacity =0.3, fill= gray] (a45)-- (b45) .. controls
  ($0.5*(a45) + 0.5*(b45)$) and ($0.5*(a135) + 0.5*(b135)$) .. (b135)
  -- (a135) .. controls ($0.5*(a135) + 0.5*(b135)$) and ($0.5*(a45) + 0.5*(b45)$) .. (a45);

    \draw (b45) .. controls 
    ($0.5*(a45) + 0.5*(b45)$) and ($0.5*(a135) + 0.5*(b135)$)
    .. (b135) coordinate[pos=0.6] (b90);
    \draw (a135) .. controls ($0.5*(a135) + 0.5*(b135)$) and
    ($0.5*(a45) + 0.5*(b45)$) .. (a45) coordinate[pos=0.58] (a90);

    \draw[very thin] (a90) -- (a270);
    \draw[very thin] (b270) -- (b90);
 
 \draw (a45) -- (a135);
 \draw (a135) -- (a225);
 \draw[double] (a225) -- (a315);
 \draw (a315) -- (a45);
 \draw[densely dotted] (b45) -- (b135);
 \draw[densely dotted] (b135) -- (b225);
 \draw[double, densely dotted] (b225) -- (b315);
 \draw (b315) -- (b45);
 \draw[double] (A45) -- (a45);
 \draw[double] (A135) -- (a135);
\draw (A225) -- (a225);
\draw (A315) -- (a315);
 \draw[double] (B45) -- (b45);

\draw (B225) -- (b225);
\draw (B315) -- (b315);
\end{scope}}}}] \\
    \eval[{\NB{\tikz[font= \tiny, scale=1]{\input{\imagesfolder/4v_square1222I}}}}] &=
    \eval[{\NB{\tikz[font= \tiny, scale=1]{\begin{scope}
  \foreach \x in {45, 135, 225, 315}{
\begin{scope}[zxplane=1, rotate=35]
    \coordinate (a\x) at (\x:0.5);
    \coordinate (A\x) at (\x:1);
  \end{scope}
  \begin{scope}[zxplane=-1, rotate =35]
    \coordinate (b\x) at (\x:0.5);
    \coordinate (B\x) at (\x:1);
  \end{scope}
}

\begin{scope}[zxplane=0.3, rotate =35]
  \coordinate (M1) at (0.354, 0);
  \coordinate (M2) at (-0.354, 0);
  \draw (M1) -- (M2);
\end{scope} 

\begin{scope}[zxplane=-0.3, rotate =35]
  \coordinate (m1) at (0.354, 0);
  \coordinate (m2) at (-0.354, 0);
  \draw (m1) -- (m2);
  \draw (m1) -- (M1);
  \draw (m2) -- (M2);
\end{scope} 

\filldraw[fill= gray, fill opacity=0.3] (a45) -- (M1) -- (m1) -- (b45) -- (B45)
-- (A45);
\filldraw[fill =gray, fill opacity=0.3] (a135) -- (M2) -- (m2) -- (b135) -- (B135)
-- (A135);

\filldraw[fill =gray, fill opacity=0.3] (a135) -- (M2) -- (a225);
\filldraw[fill =gray, fill opacity=0.3] (a45) -- (M1) -- (a315);

\filldraw[fill =gray, fill opacity=0.3] (b135) -- (m2) -- (b225);
\filldraw[fill =gray, fill opacity=0.3] (b45) -- (m1) -- (b315);

\draw (a225) -- (A225) -- (B225) -- (b225);
\draw (a315) -- (A315) -- (B315) -- (b315);
\draw (a315) -- (a225);

\draw[double] (a315)-- (a45) --(a135) --(a225);
\draw[double] (a45) -- (A45);
\draw[double] (a135) --(A135);

\draw[double] (b315)-- (b45);
\draw[double, densely dotted] (b45) --(b135);
\draw[double, densely dotted] (b135) --(b225);
\draw[double] (b45) -- (B45);
\draw[double, densely dotted] (b135) --(B135);
\draw[densely dotted] (b225) -- (b315);


\end{scope}}}}] +
    \eval[{\NB{\tikz[font= \tiny, scale=1]{\begin{scope}
  \foreach \x in {45, 135, 225, 315}{
\begin{scope}[zxplane=1, rotate=35]
    \coordinate (a\x) at (\x:0.5);
    \coordinate (A\x) at (\x:1);
  \end{scope}
  \begin{scope}[zxplane=-1, rotate =35]
    \coordinate (b\x) at (\x:0.5);
    \coordinate (B\x) at (\x:1);
  \end{scope}

  \begin{scope}[zxplane=-0.4, rotate =35]
    \coordinate (n\x) at (\x:0.5);
  \end{scope}

  \begin{scope}[zxplane= 0.4, rotate =35]
    \coordinate (m\x) at (\x:0.5);
  \end{scope}

}

\draw (n45) -- (n135) -- (n225) -- (n315)--cycle;

\draw (m45) -- (m135) -- (m225) -- (m315)--cycle;
\draw (m45) -- (m135) -- (n135) -- (n45)--cycle;
\draw (b315) -- (a315) -- (a225) -- (b225);

\filldraw[fill opacity=0.3, fill= gray ] (a45) -- (A45) -- (B45) -- (b45);
\filldraw[fill opacity=0.3, fill = gray] (a135) -- (A135) -- (B135) -- (b135);
\draw (a225) -- (A225) -- (B225) -- (b225);
\draw (a315) -- (A315) -- (B315) -- (b315);
\filldraw[fill opacity=0.3, fill = gray]  (m135) -- (a135) --
(a45)-- (m45);
\filldraw[fill opacity=0.3, fill = gray]  (n135) -- (b135) --
(b45)-- (n45);
\filldraw[fill opacity=0.3, fill = gray] (m225) -- (a225) --
(a135)--(m135);
\filldraw[fill opacity=0.3, fill = gray] (b225) -- (n225) --
(n135)--(b135);
\filldraw[fill opacity=0.3, fill = gray] (n225) -- (n315) -- (m315) --
(m225);
\filldraw[fill opacity=0.3, fill = gray] (b45) -- (n45) -- (n315) --
(b315);
\filldraw[fill opacity=0.3, fill = gray] (m45) -- (a45) -- (a315) --
(m315);

\draw[double] (a315)-- (a45) --(a135) --(a225);
\draw[double] (a45) -- (A45);
\draw[double] (a135) --(A135);

\draw[double] (b315)-- (b45);
\draw[double, densely dotted] (b45) --(b135);
\draw[double, densely dotted] (b135) --(b225);
\draw[double] (b45) -- (B45);
\draw[double, densely dotted] (b135) --(B135);
\draw[densely dotted] (b225) -- (b315);


\end{scope}
}}}] \\
     \label{eq_rel_12}
    \eval[{\NB{\tikz[font= \tiny, scale=1]{\input{\imagesfolder/4v_square1122I}}}}] &=
    \eval[{\NB{\tikz[font= \tiny, scale=1]{\begin{scope}
  \foreach \x in {45, 135, 225, 315}{
\begin{scope}[zxplane=1, rotate=35]
    \coordinate (a\x) at (\x:0.5);
    \coordinate (A\x) at (\x:1);
  \end{scope}
  \begin{scope}[zxplane=-1, rotate =35]
    \coordinate (b\x) at (\x:0.5);
    \coordinate (B\x) at (\x:1);
  \end{scope}
}

\begin{scope}[zxplane=0.3, rotate =35]
  \coordinate (M1) at (0.354, 0);
  \coordinate (M2) at (-0.354, 0);
  \draw (M1) -- (M2);
\end{scope} 

\begin{scope}[zxplane=-0.3, rotate =35]
  \coordinate (m1) at (0.354, 0);
  \coordinate (m2) at (-0.354, 0);

\end{scope} 


\filldraw[fill opacity=0.3, fill= gray ] (a315) -- (A315) -- (B315)
-- (b315) -- (m1) -- (M1) -- cycle;
\filldraw[fill opacity=0.3, fill = gray] (M2) -- (a135) -- (A135) --
(B135) -- (b135) -- (m2);
\draw (a225) -- (A225) -- (B225) -- (b225);
\draw (a45) -- (A45) -- (B45) -- (b45);
\filldraw[fill opacity=0.3, fill = gray]  (M2) -- (a135) -- (a45)--
(M1);
\filldraw[fill opacity=0.3, fill = gray]  (b135) -- (m2) -- (m1)--
(b45);


  \draw (m2) -- (M2);
  \draw (b225) -- (m2);
  \draw (b135) -- (m2);

\filldraw[fill opacity=0.3, fill = gray] (b225) -- (m2) -- (b135) --cycle;
\filldraw[fill opacity=0.3, fill = gray] (a225) -- (M2) --  (a135) -- cycle;

\draw[orange, thick]  ($0.3*(a135) +0.7*(B135)$) .. controls +(0, 0.2) and
+(0, -0.2)..   ($0.7*(A135)
+0.3*(b135)$) .. controls +(0, 0.2) and
+(0, -0)..  ($0.5*(a135) + 0.5*(M2)$)--(M1) .. controls
($0.3*(A315)+0.3*(B315)$) .. (m1) -- ($0.7*(b135) + 0.3*(m2)$)
.. controls +(0, 0) and +(0, -0.2)..  cycle;


\draw (a225) -- (a315)-- (a45);
\draw[double] (a45) --(a135) --(a225);
\draw[double] (a135) --(A135);

  \draw (m1) -- (m2);
  \draw (m1) -- (M1);
  \draw (a45) -- (M1);
  \draw (a315) -- (M1);
  \draw (a225) -- (M2);
  \draw (a135) -- (M2);
  \draw (b45) -- (m1);
  \draw (b315) -- (m1);

\draw (b315)-- (b45);
\draw[double, densely dotted] (b45) --(b135);
\draw[double, densely dotted] (b135) --(b225);
\draw (b45) -- (B45);
\draw[double, densely dotted] (b135) --(B135);
\draw[double] (b315) --(B315);
\draw[double] (a315) --(A315);
\draw[densely dotted] (b225) -- (b315);


\end{scope}}}}] +
    \eval[{\NB{\tikz[font= \tiny, scale=1]{\begin{scope}
  \foreach \x in {45, 135, 225, 315}{
\begin{scope}[zxplane=1, rotate=35]
    \coordinate (a\x) at (\x:0.5);
    \coordinate (A\x) at (\x:1);
  \end{scope}
  \begin{scope}[zxplane=-1, rotate =35]
    \coordinate (b\x) at (\x:0.5);
    \coordinate (B\x) at (\x:1);
  \end{scope}
}

\begin{scope}[zxplane=0.3, rotate =35]
  \coordinate (M1) at (0, 0.3540);
  \coordinate (M2) at (0, -0.354);
  \draw (M1) -- (M2);
\end{scope} 

\begin{scope}[zxplane=-0.3, rotate =35]
  \coordinate (m1) at (0, 0.354);
  \coordinate (m2) at (0, -0.354);

\end{scope} 


\filldraw[fill opacity=0.3, fill= gray ] (a315) -- (A315) -- (B315)
-- (b315) -- (m2) -- (M2) -- cycle;
\filldraw[fill opacity=0.3, fill = gray] (M1) -- (a135) -- (A135) --
(B135) -- (b135) -- (m1);
\draw (a225) -- (A225) -- (B225) -- (b225);
\draw (a45) -- (A45) -- (B45) -- (b45);
\filldraw[fill opacity=0.3, fill = gray]  (M1) -- (a135) -- (a225)--
(M2);
\filldraw[fill opacity=0.3, fill = gray]  (b135) -- (m1) -- (m2)--
(b225);


\filldraw[fill opacity=0.3, fill = gray] (b45) -- (m1) -- (b135) --cycle;
\filldraw[fill opacity=0.3, fill = gray] (a45) -- (M1) --  (a135) -- cycle;

\draw[orange, thick]  ($0.3*(a135) +0.7*(B135)$) .. controls +(0, 0.2) and
+(0, -0.2)..   ($0.7*(A135)
+0.3*(b135)$) .. controls +(0, 0.2) and
+(0, -0)..  ($0.5*(a135) + 0.5*(M1)$)--(M2) .. controls
($0.5*(A315)+0.5*(B315)$) .. (m2) -- ($0.7*(b135) + 0.3*(m1)$)
.. controls +(0, 0) and +(0, -0.2)..  cycle;


\draw (a225) -- (a315)-- (a45);
\draw[double] (a45) --(a135) --(a225);
\draw[double] (a135) --(A135);

\begin{scope}
  \draw (m1) -- (m2);
  \draw (m1) -- (M1);
  \draw (m2) -- (M2);
  \draw (b225) -- (m2);
  \draw (b315) -- (m2);

  \draw (a225) -- (M2);
  \draw (a315) -- (M2);
  \draw (a45) -- (M1);
  \draw (a135) -- (M1);
  \draw (b45) -- (m1);
  \draw (b135) -- (m1);
\end{scope}[red]
  
\draw (b315)-- (b45);
\draw[double, densely dotted] (b45) --(b135);
\draw[double, densely dotted] (b135) --(b225);
\draw (b45) -- (B45);
\draw[double, densely dotted] (b135) --(B135);
\draw[double] (b315) --(B315);
\draw[double] (a315) --(A315);
\draw[densely dotted] (b225) -- (b315);


\end{scope}}}}] \\
    \label{eq_rel_9}
    \eval[{\NB{\tikz[font= \tiny, scale=1]{\input{\imagesfolder/4v_square1212I}}}}] &= \eval[{\NB{\tikz[font= \tiny, scale=1]{\begin{scope}
  \foreach \x in {270, 0, 90, 180}{
\begin{scope}[zxplane=1, rotate=0]
    \coordinate (a\x) at (\x:0.5);
    \coordinate (A\x) at (\x:1);
  \end{scope}
  \begin{scope}[zxplane=-1, rotate =0]
    \coordinate (b\x) at (\x:0.5);
    \coordinate (B\x) at (\x:1);
  \end{scope}
\draw (a\x) -- (A\x) -- (B\x) -- (b\x);
}

  \begin{scope}[zxplane=-0.9, rotate =0]
    \coordinate (b1) at (80:0.5);
    \coordinate (b2) at (10:0.5);
  \end{scope}

\begin{scope}[zyplane = 0]

 \fill[fill opacity=0.0, red, draw=black] (a0) .. controls +(0,-0.8) and +(0, -0.8).. (a270) coordinate [pos=0.4] (al);

  \fill[fill opacity=0.0, red, draw=black] (a90) .. controls +(0,-0.8) and +(0, -0.8).. (a180) coordinate[pos=0.4] (ar);
 \fill[fill opacity=0.3, gray, draw=black] 
 (b0) .. controls +(0,0.8) and +(0, 0.8).. (b270) coordinate[pos=0.6] (bl);
  \fill[fill opacity=0.0, red, draw=black] 
 (b90) coordinate[pos=0.67] (b45)  .. controls +(0,0.8) and +(0,
 0.8).. (b180) coordinate[pos=0.6] (br);
 \draw[very thin] (bl)--(br);
 \draw[very thin] (al)--(ar);

\end{scope}

 \begin{scope}[zxplane=0, rotate =0]
\draw[very thin, dashed] (90:1) .. controls (90:0.5) and (180:0.5)
.. (180:1);
\draw[very thin, dashed] (270:1) .. controls (270:0.5) and (0:0.5) .. (0:1);
 \end{scope}

 \draw[double] (a90) -- (a0);
 \draw[double] (a270) -- (a180);
 \draw[double, densely dotted] (b270) -- (b180);
 \draw[double] (b0) -- (b90);
\draw[densely dotted] (b90) -- (b180);
\draw[densely dotted] (b0) -- (b270) coordinate[pos=0.6] (b315);
\draw (b270)--(b315);
\draw (a0) -- (a270);
\draw (a90) -- (a180);

 \begin{scope}[zxplane=-0.85, rotate =0]
    \coordinate (d1) at (45:0.354);
  \end{scope}

  \begin{scope}[zxplane=0.89, rotate =0]
    \coordinate (d2) at (225:0.354);
  \end{scope}

\fill[fill = gray, fill opacity=0.3] (b180)-- (b270) .. controls
+(0,0.4,0) and +(-0.03,0, 0) .. (bl) -- (br) .. controls +(-0.03,0, 0)  and +(0,0.4,0) 
.. cycle ;

\fill[fill = gray, fill opacity=0.3] (b90)-- (b0) .. controls
+(0,0.5,0) and +(0.11,0, 0) .. (bl) -- (br) .. controls +(0.11,0, 0)  and +(0,0.5,0) 
.. cycle ;

\fill[fill = gray, fill opacity=0.3] (a180)-- (a270) .. controls
+(0,-0.5,0) and +(-0.11,0, 0) .. (al) -- (ar) .. controls +(-0.11,0, 0)  and +(0,-0.5,0) 
.. cycle ;

\fill[fill = gray, fill opacity=0.3] (a90)-- (a0) .. controls
+(0,-0.4,0) and +(0.03,0, 0) .. (al) -- (ar) .. controls +(0.03,0, 0)  and +(0,-0.4,0) 
.. cycle ;



\end{scope}}}}] +
    \eval[{\NB{\tikz[scale=1]{\begin{scope}
  \foreach \x in {270, 0, 90, 180}{
\begin{scope}[zxplane=1, rotate=0]
    \coordinate (a\x) at (\x:0.5);
    \coordinate (A\x) at (\x:1);
  \end{scope}
  \begin{scope}[zxplane=-1, rotate =0]
    \coordinate (b\x) at (\x:0.5);
    \coordinate (B\x) at (\x:1);
  \end{scope}
\draw (a\x) -- (A\x) -- (B\x) -- (b\x);
}

  \begin{scope}[zxplane=-0.9, rotate =0]
    \coordinate (b1) at (80:0.5);
    \coordinate (b2) at (10:0.5);
  \end{scope}

\begin{scope}[zyplane = 0]

 \fill[fill opacity=0.3, gray, draw=black] (a270)
 --(a180) .. controls +(0,-0.8) and +(0, -0.8).. (a270) coordinate [pos=0.65] (al);

  \fill[fill opacity=0.3, gray, draw=black] (a0)
 --(a90) .. controls +(0,-0.8) and +(0, -0.8).. (a0) coordinate[pos=0.65] (ar);
 \fill[fill opacity=0.3, gray, draw=black] 
 (b180) .. controls +(0,0.8) and +(0, 0.8).. (b270) coordinate[pos=0.35] (bl);
  \fill[fill opacity=0.3, gray, draw=black] 
 (b90) coordinate[pos=0.67] (b45)  .. controls +(0,0.8) and +(0,
 0.8).. (b0) coordinate[pos=0.35] (br);
 \draw[very thin] (bl)--(br);
 \draw[very thin] (al)--(ar);

\end{scope}

 \begin{scope}[zxplane=0, rotate =0]
\draw[very thin, dashed] (90:1) .. controls (90:0.5) and (0:0.5)
.. (0:1);
\draw[very thin, dashed] (270:1) .. controls (270:0.5) and (180:0.5) .. (180:1);
 \end{scope}

 \draw[double] (a90) -- (a0);
 \draw[double] (a270) -- (a180);
 \draw[double, densely dotted] (b270) -- (b180);
 \draw[double] (b0) -- (b90);
\draw[densely dotted] (b90) -- (b180);
\draw[densely dotted] (b0) -- (b270) coordinate[pos=0.6] (b315);
\draw (b270)--(b315);
\draw (a0) -- (a270);
\draw (a90) -- (a180);

 \begin{scope}[zxplane=-0.85, rotate =0]
    \coordinate (d1) at (45:0.354);
  \end{scope}

  \begin{scope}[zxplane=0.89, rotate =0]
    \coordinate (d2) at (225:0.354);
  \end{scope}

      \draw[thick, orange] (b1) .. controls +(0,0.55) and +(0, 0.55)
    .. (b2) --cycle;

\node[scale = 0.7] at (d1) {$\bullet$};
\node[yshift=-0.35cm, scale = 0.7] at (d1) {$e_1$};

\end{scope}}}}]   +
    \eval[{\NB{\tikz[ scale=1]{\begin{scope}
  \foreach \x in {270, 0, 90, 180}{
\begin{scope}[zxplane=1, rotate=0]
    \coordinate (a\x) at (\x:0.5);
    \coordinate (A\x) at (\x:1);
  \end{scope}
  \begin{scope}[zxplane=-1, rotate =0]
    \coordinate (b\x) at (\x:0.5);
    \coordinate (B\x) at (\x:1);
  \end{scope}
\draw (a\x) -- (A\x) -- (B\x) -- (b\x);
}

  \begin{scope}[zxplane=-0.9, rotate =0]
    \coordinate (b1) at (80:0.5);
    \coordinate (b2) at (10:0.5);
  \end{scope}

\begin{scope}[zyplane = 0]

 \fill[fill opacity=0.3, gray, draw=black] (a270)
 --(a180) .. controls +(0,-0.8) and +(0, -0.8).. (a270) coordinate [pos=0.65] (al);

  \fill[fill opacity=0.3, gray, draw=black] (a0)
 --(a90) .. controls +(0,-0.8) and +(0, -0.8).. (a0) coordinate[pos=0.65] (ar);
 \fill[fill opacity=0.3, gray, draw=black] 
 (b180) .. controls +(0,0.8) and +(0, 0.8).. (b270) coordinate[pos=0.35] (bl);
  \fill[fill opacity=0.3, gray, draw=black] 
 (b90) coordinate[pos=0.67] (b45)  .. controls +(0,0.8) and +(0,
 0.8).. (b0) coordinate[pos=0.35] (br);
 \draw[very thin] (bl)--(br);
 \draw[very thin] (al)--(ar);

\end{scope}

 \begin{scope}[zxplane=0, rotate =0]
\draw[very thin, dashed] (90:1) .. controls (90:0.5) and (0:0.5)
.. (0:1);
\draw[very thin, dashed] (270:1) .. controls (270:0.5) and (180:0.5) .. (180:1);
 \end{scope}

 \draw[double] (a90) -- (a0);
 \draw[double] (a270) -- (a180);
 \draw[double, densely dotted] (b270) -- (b180);
 \draw[double] (b0) -- (b90);
\draw[densely dotted] (b90) -- (b180);
\draw[densely dotted] (b0) -- (b270) coordinate[pos=0.6] (b315);
\draw (b270)--(b315);
\draw (a0) -- (a270);
\draw (a90) -- (a180);

 \begin{scope}[zxplane=-0.85, rotate =0]
    \coordinate (d1) at (45:0.354);
  \end{scope}

  \begin{scope}[zxplane=0.89, rotate =0]
    \coordinate (d2) at (225:0.354);
  \end{scope}


\node[scale = 0.7] at (d2) {$\bullet$};
\node[yshift = 0.3cm, scale = 0.7] at (d2) {$e_1$};

\end{scope}}}}] 
  \end{align}
\end{lemma}

\begin{remark}
    Relations~\eqref{eq_rel_A1}--\eqref{eq_rel_6} have the form $\id_{\Gamma}=\iota\circ p$, where $\Gamma$ is the web at the top and bottom of the foams and the RHS of the relation factors into foams $\iota$ and $p$ between webs $\Gamma$ and $\Gamma'$, so that the RHS foam can be written as the composition $\Gamma\stackrel{p}{\lra}\Gamma'\stackrel{\iota}{\lra}\Gamma$. Web $\Gamma'$ is the middle cross-section of the foam on the right of the equation. Furthermore, 
    \[p\circ \iota=\id_{\Gamma'},
    \]
    which in each case can be checked by a direct computation. The foam $p\circ \iota$ is given by stacking the bottom half of the RHS foam above the top half of the RHS foam in eaach of these equations. For relations~\eqref{eq_rel_A2} and~\eqref{eq_rel_13} there exists a homeomorphism of webs $\Gamma\cong \Gamma'$ that extends to a homeomorphism of foams $\iota$ and $p$ allowing to derive the relation $p\circ\iota=\id_{\Gamma'}$ from $\iota\circ p =\id_{\Gamma}$. For equation~\eqref{eq_rel_A2} this homeomorphism reverses the orientation of $\R^2$ and $\R^2\times [0,1]$ (foam evaluation is invariant under orientation reversal in  $\R^2\times [0,1]$). 

    Likewise, relations~\eqref{eq_rel_2},~\eqref{eq_rel_3} in Lemma~\ref{lemma_rel_1} and relations~\eqref{eq_rel_7},~\eqref{eq_rel_8} in Lemma~\ref{lemma_rel_2} have the form $\id_{\Gamma}=\iota_1\circ p_1+\iota_2\circ p_2$, where $\Gamma\stackrel{p_i}{\lra}\Gamma'\stackrel{\iota_i}{\lra}\Gamma$ is the factorization of the $i$-th foam on the RHS of the relation, $i=1,2$. Relations 
    \[p_j \circ \iota_i = \delta_{i,j}\id_{\Gamma'}, \ \ i,j\in\{1,2\} 
    \]
    hold (direct computation).  

Relations~\eqref{eq_rel_10}--\eqref{eq_rel_12} in Lemma~\ref{lemma_rel_5} have the form 
\[\id_{\Gamma}=\iota_1\circ p_1+\iota_2\circ p_2,
\]
where $\Gamma\stackrel{p_i}{\lra}\Gamma_i\stackrel{\iota_i}{\lra}\Gamma$ is the factorization of the $i$-th foam on the RHS of the relation, $i=1,2$. Relations 
    \[p_j \circ \iota_i = \delta_{i,j}\id_{\Gamma_i}, \ i,j\in\{1,2\}. 
    \]
hold (direct computation). 
    Relation~\eqref{eq_rel_9} in Lemma~\ref{lemma_rel_5}
    has a similar form, 
    \[\id_{\Gamma}=\iota_1\circ p_1+\iota_2\circ p_2+\iota_3\circ p_3,
    \] and, additionally, relations $p_j \circ \iota_i = 0, \ j\not= i, \ \ p_i\circ \iota_i =\id_{\Gamma_i}$
    hold. 
\end{remark}

\section{State spaces}
\label{sec_state_sp}

Define the degree of a foam $U: \Gamma_0 \to \Gamma_1$ by the same formula as in Definition~\ref{def:degreeF}. 
If a foam $F$ with boundary and no dots admits a coloring $c$, formula~\eqref{eq:degfromsurfaces} holds for it as well. Presence of dots increases the degree of the foam by the sum of degrees of the dots. 

\begin{proposition} \label{prop_degree} For composable foams $U$ and $V$, 
\[\deg(UV) = \deg(U) + \deg(V).\] 
\end{proposition}

\begin{proof}
    The proof is analogous to that of \cite[Proposition~3.1]{KhovRob1}. 
\end{proof}

{\it State space of a foam.}
Following \cite{KhovRob1}, we define the state space of a web, see also \cite{BHMV,Kho04,RobWag} for related constructions and \cite{KhoKit} for a discussion. 

\begin{definition}
Given a web $\Gamma$, its state space $\functor[\Gamma]$ 
is a $\Z$-graded $R$-module generated by symbols $\functor[U]$, for all foams $U$ from the empty web $\emptyset_1$ to $\Gamma$.   A relation $\Sigma_i a_i \functor[U_i] = 0$ for $a_i \in R$ and $U_i \in \operatorname{Hom}_{\catFoam}(\emptyset_1, \Gamma)$ holds in $\functor[\Gamma]$ if and only if $\Sigma_i a_i \eval[VU_i] = 0$ for any foam $V\in \operatorname{Hom}_{\catFoam}(\Gamma,\emptyset_1)$.
\end{definition}

\begin{theorem}\label{thm_functor}
    A foam $F$ determines a graded $R$-module homomorphism 
    \begin{equation*}
        \functor(F) \ : \ \functor(\partial_0 F)\lra \functor(\partial_1 F) 
    \end{equation*}
    of degree $\deg(F)$. These maps fit together into a lax monoidal functor 
    \begin{equation*}
    \functor \ : \ \catFoam \lra R\mathsf{-gmod}
        \end{equation*}
    from the category of foams to the category of graded $R$-modules and homogeneous module homomorphisms. 
\end{theorem}
We say that a graded module homomorphism is \emph{homogeneous of degree $i$} if it changes the degree of all homogeneous elements by $i$. 

\begin{proposition}
    Graded $R$-module $\functor(\Gamma)$ is finitely-generated, for any unoriented $SL(4)$ web $\Gamma$. 
\end{proposition}

\begin{proof}
    The proof follows by directly extending the proof of Proposition~3.9 in~\cite{KhovRob1}  (including Proposition~3.6 and Corollary~3.7 there) from three to four colors. We omit the details. 
\end{proof}

There is a canonical isomorphism $\functor \left(
  \emptyset_1\right)\cong \ourring \emptyset_2$ of $R$-bimodules. Here
$\emptyset_1$ is the empty web in $\R^2$ and $\emptyset_2$ is the
empty foam in $\R^2\times [0,1]$, thought of as a degree $0$ generator
of the rank one free $R$-module $\functor(\emptyset_1)$.

The lemmas and remarks of Section~\ref{sec:neck-cutt-relat} translate
into:

\begin{theorem} \label{thm_dir_sum_rels} Relations in Proposition \ref{prop:relations}  can be lifted to the isomorphisms below. 
\begin{align}
\label{eq_cong_circles} 
    \functor \left(\NB{\tikz[]{}} \right) &\cong [4] , &&& 
    \functor\left(\NB{\tikz[]{}}\right) &\cong \frac{[3][4]}{[2]} , \\
\label{eq_liftloop_zero}
    \functor\left(\NB{\tikz[]{}}\right) &= 0, &&&
\functor\left(\NB{\tikz[]{}}\right) &= 0, \\ 
\label{eq_liftdigon_1}
    \functor \left(\NB{\tikz[]{}}\right) & \cong 
[2]\, \functor\left(\NB{\tikz[]{}}\right), &&&
    \functor\left(\NB{\tikz[]{}}\right) &\cong
    [2]\, \functor\left(\NB{\tikz[]{}}\right), \\
\label{eq_liftdigon_2}
    \functor\left(\NB{\tikz[]{}}\right) &\cong
    [3]\, \functor\left(\NB{\tikz[]{}}\right), &&& \functor\left(\NB{\tikz[]{}}\right) &=
    0, 
    \\
\label{eq_liftdefects}
    \functor\left(\NB{\tikz[]{}}\right) &\cong
    \functor\left(\NB{\tikz[]{}}\right), &&& \functor\left(\NB{\tikz[]{}}\right) &\cong \functor\left(\NB{\tikz[]{}}\right), \\
\label{eq_liftquadruple_v}\functor\left( \NB{\tikz[]{\input{\imagesfolder/4v_dumbell-bullet}}}\right) &\cong
    \functor\left(\NB{\tikz[rotate=90]{\input{\imagesfolder/4v_dumbell-bullet}}}\right) \\
\label{eq_lifttriangle_1}
\functor\left(\NB{\tikz[]{}}\right) & \cong \functor\left(\NB{\tikz[]{}}\right),   &&& \functor\left(\NB{\tikz[]{}}\right) & \cong 2\, \functor\left(\NB{\tikz[]{}}\right),   \\
\label{eq_lifttriangle_2}
\functor\left(\NB{\tikz[]{}}\right) &\cong \functor\left(\NB{\tikz[]{}}\right),   &&&  \functor\left(\NB{\tikz[]{}}\right) &\cong [2]\,\functor\left(\NB{\tikz[]{}}\right),  
\end{align}\vspace{-0.35cm}
\begin{align}
\label{eq_liftsquare_1}
\functor\left( \NB{\tikz[]{}}\right)&\cong
    \functor\left( \NB{\tikz[]{}}\right) \oplus
    \functor\left( \NB{\tikz[]{}}\right),\\
    \label{eq_liftsquare_2}
\functor\left( \NB{\tikz[]{}}\right)&\cong
    \functor\left( \NB{\tikz[]{}}\right) \oplus
    \functor\left( \NB{\tikz[]{\input{\imagesfolder/12v12}}}\right), \\
    \label{eq_liftsquare_3}
    \functor\left( \NB{\tikz[]{}}\right)&\cong
    \functor\left( \NB{\tikz[rotate=90]{}}\right) \oplus
    [2]\, \functor\left( \NB{\tikz[]{}}\right), \\
    \label{eq_liftsquare_4}
    \functor\left( \NB{\tikz[]{\input{\imagesfolder/Square2211}}}\right) &\cong
    \functor\left( \NB{\tikz[rotate=90]{}}\right) \oplus
    \functor\left( \NB{\tikz[]{}}\right), 
\end{align}
\end{theorem}

The first term on the RHS of~\eqref{eq_liftsquare_2} can be further simplified via~\eqref{eq_liftsquare_1}. 
Note that the last relation~\eqref{eq_sq5} in Proposition~\ref{prop:relations} is excluded from the above proposition---the authors did not find the corresponding direct sum decomposition for state spaces. Quantum integer 
\[
[n]:= \frac{q^n-q^{-n}}{q-q^{-1}}= q^{n-1}+q^{n-3}+\dots +q^{1-n}. 
\]
Powers of $q$ correspond to grading shifts, with the notation 
\[
f(q)V\ := \ \oplus_{i\in \Z} V^{a_i}\{i\}, \ \ f(q)=\sum_i a_i q^i, \ a_i\in \mathbb{N}, 
\]
for the sum of copies of a $\Z$-graded module $V$ with multiplicities given by coefficients of the polynomial $f(q)\in\mathbb{N}[q,q^{-1}]$.

Relations \eqref{eq_cong_circles} say that the state space of a web $\Gamma$ with a thin, respectively thick, innermost circle is isomorphic to 
four, respectively six copies of the state space of that web without the circle, placed in degrees $[4]=q^3+q+q^{-1}+q^{-3}$ and $[3][4]/[2]=(q^2+1+q^{-2})(q^2+q^{-2})$.

\begin{remark}
State spaces of thin and thick circles $\tau(\NB{\tikz[]{}})$ and $\tau(\NB{\tikz[]{}})$ are  commutative Frobenius $R$-algebras which are isomorphic to $U(4)$-equivariant cohomology algebras, over $\mathbb{F}_2$, of the complex 
projective space $\mathbb{CP}^3$ and the complex Grassmannian $\mathsf{Gr}(2,2)$ of two-dimensional subspaces in $\mathbb{C}^4$. The ground ring $R$ is isomorphic to $U(4)$-equivariant cohomology of a point $p$: 
\[
R\cong \mathsf{H}^{\ast}_{U(4)}(p,\mathbb{F}_2), \ \ \ 
\functor(\NB{\tikz[]{}}) \cong \mathsf{H}^{\ast}_{U(4)}(\mathbb{CP}^3,\mathbb{F}_2), \ \ \ 
\functor(\NB{\tikz[]{}}) \cong \mathsf{H}^{\ast}_{U(4)}(\mathsf{Gr}(2,2),\mathbb{F}_2). 
\]
Frobenius algebra structure comes from the functor $\functor$ applied to unknotted surface cobordisms in $\R^2\times [0,1]$ between unions of innermost circles. 

Furthermore, equivariant cohomology of partial and full flag varieties
\[
\mathsf{Fl}_{12}=\{L_1\subset L_2\subset \C^4|\dim L_i=i\}, \ \ 
\mathsf{Fl}_{123}= \{L_1\subset L_2\subset L_3\subset \C^4|\dim L_i=i\}
\]
can be identified with state spaces of the 112 theta-web
$\NB{\tikz[scale=0.35]{\begin{scope}[rotate=90]
  \draw[double] (0,0) -- (1,0);
  \draw (0.5,0) circle (0.5cm and 0.7cm);
\end{scope}}} $ and the state space of the
web \NB{\tikz[scale=0.35]{\begin{scope}[xscale=1.3]
  \draw (1,1) arc (90:-90:0.5) -- cycle;
  \draw (0,0) arc (270:90:0.5) --cycle;
  \draw[double] (0,1) -- (1,1) node[pos=0.5,yshift=-0.1mm] {\myvertex};
  \draw[double] (0,0) -- (1,0) node[pos=0.5,yshift=-0.1mm] {\myvertex};
\end{scope}}}.\[
\functor[{\NB{\tikz[scale=0.35]{}}}] \cong \mathsf{H}^{\ast}_{U(4)}(\mathsf{Fl}_{12},\mathbb{F}_2), \ \ \ 
\functor[{\NB{\tikz[scale=0.35]{}}}] \cong \mathsf{H}^{\ast}_{U(4)}(\mathsf{Fl}_{123},\mathbb{F}_2). 
\]

\end{remark}

Introduce a rather naive complexity function on webs by assigning to a web $\Gamma$ the number $n_e(\Gamma)$  of edges of $\Gamma$. Circles count as edges. The LHS to RHS transformations in Theorem~\ref{thm_dir_sum_rels} do not increase this complexity function and strictly reduce it in all equations except \eqref{eq_liftquadruple_v}, the second equation in \eqref{eq_liftdefects}, and using \eqref{eq_liftsquare_1} to further  reduce \eqref{eq_liftsquare_2}. 

\begin{definition}
   Let $\Webs$ be the set of webs up to planar isotopy. We define \emph{reducible} finite subsets $S\subset \Webs$ inductively as follows: 
   \begin{itemize}
   \item The empty subset is reducible. 
   \item The subset $\{\emptyset_1\}$ which consists of the empty web
     is reducible.
   \item Any  non-colorable web ($\ncol=0$) is
     reducible.
   \item If $S$ is reducible and $T\subseteq S$, $T$ is reducible. 
   \item If $S$ and $T$ are reducible then $S\cup T$ is reducible.
   \item Given a finite subset $S\subset \Webs$ and a graph $\Gamma\subset S$, suppose a portion of $\Gamma$ is as shown on the LHS of one of the equations in Theorem~\ref{thm_dir_sum_rels}. Let $s(\Gamma)$ be the subset of graphs on the RHS of the corresponding formula. If RHS is $0$, set $s(\Gamma)$ to be the empty subset. If $(S\setminus \{\Gamma\}) \cup s(\Gamma)$ is reducible, then $S$ is reducible. 
   \end{itemize}
We say that $\Gamma$ is reducible if the set $\{\Gamma\}$ is reducible.
\end{definition}

The idea for this definition is to keep reducing a web by going from a web in the LHS of one of the relations in Theorem~\ref{thm_dir_sum_rels} to the set of webs on the RHS of these relations until each web is reduced to either the empty web or some webs having no colorings for obvious reasons (see for example \eqref{eq_liftloop_zero}) and can be removed. 

\vspace{0.07in} 

Suppose that $\functor(\Gamma)$ is a free graded $R$-module. Define 
\begin{equation*}
    t_q(\Gamma) \ := \ \mathrm{grank}(\functor(\Gamma)) \ \in \N[q,q^{-1}]. 
\end{equation*}
to be the graded rank of $\functor(\Gamma)$

\begin{proposition}
    Suppose the web $\Gamma$ is reducible. Then the state space $\functor[\Gamma]$ is a free graded $R$-module. Its graded rank $t_q(\Gamma)$  can be computed inductively via Theorem~\ref{thm_dir_sum_rels} by replacing direct sums with sums. Thus, isomorphisms of state spaces in that proposition descend to skein relations on $t_q$. 
  \end{proposition}

  \begin{example}
    \[
      t_q\left( \, \NB{\tikz[scale=0.7]{\begin{scope}
  \foreach \x in{45, 135, 225, 315}{
    \coordinate (a\x) at (\x:0.5);
    \coordinate (A\x) at (\x:1);
  }
    \draw[double] (a225) -- (a135);
    \draw[double] (A225) -- (A135);
    \draw[double] (a45) -- (A45);
    \draw[double] (a315) -- (A315);
  \draw (A45) -- (A135) -- (a135) -- (a45) -- (a315) -- (a225) --
  (A225) -- (A315) -- cycle;
\end{scope}}} \, \right) =
      t_q\left( \, \NB{\tikz[scale=0.7]{\begin{scope}
  \foreach \x in{45, 135, 225, 315}{
    \coordinate (a\x) at (\x:0.5);
    \coordinate (A\x) at (\x:1);
  }
    \draw[double] (a45) -- (A45);
    \draw[double] (a315) -- (A315);
  \draw (A45) -- (A135) -- (a135) -- (a45) -- (a315) -- (a225) --
  (A225) -- (A315) -- cycle;
\end{scope}}} \, \right) +
      [2]t_q\left( \, \NB{\tikz[scale=0.7]{\begin{scope}
  \foreach \x in{45, 135, 225, 315}{
    \coordinate (a\x) at (\x:0.5);
    \coordinate (A\x) at (\x:1);
  }
   \draw (a225) -- (a135);
   \draw (A225) -- (A135);
    \draw[double] (a45) -- (A45);
    \draw[double] (a315) -- (A315);
    \draw (A45) -- (A135);
    \draw (a135) -- (a45) -- (a315) -- (a225);
    \draw (A225) -- (A315) -- (A45);
\end{scope}}} \, \right)
      =[3][4]\left([3] + [2]^2\right)
    \]
  \end{example}

For unoriented reducible $SL(3)$ webs the corresponding quantum invariant was considered by Mrudul~\cite{Mrudul} who, in particular, gave high lower bounds for the ranks of $\Z[q,q^{-1}]$-modules of such webs with $n\le 8$ boundary points. The relations on webs with boundary appear via the bilinear pairing on these modules given by computing the quantum invariants on closed webs glued from a pair of diagrams with the same boundary.   

\begin{remark} It is straightforward to extend the invariant $t_q(\Gamma)$ to all webs $\Gamma$. To do so, take the finitely-generated graded $R$-module $\functor(\Gamma)$ and choose a finite length resolution $\functor^{\bullet}(\Gamma)$ of it by finitely-generated graded $R$-modules. This is possible since graded ring $R$ has finite homological dimension (equal to 4). Define $t_q(\Gamma)$ as the graded Euler characteristic of that resolution: 
\[
t_q(\Gamma) \ := \ \sum^n_{i=0}(-1)^i \mathrm{grank}(\tau^{-i}(\Gamma)) \in \Z[q,q^{-1}], 
\]
where the resolution lies in degrees $-n, -n+1,\dots, 0$. 
The main problem with this definition and, more generally, with understanding $\functor(\Gamma)$ is the absence of an algorithm or results to determine $\functor[\Gamma]$ for non-reducible webs. In fact, we can not identify the $R$-module $\functor[\Gamma]$ even for a single non-reducible web $\Gamma$. A related question is whether $\functor[\Gamma]$ is a free graded $R$-module for any web $\Gamma$. 
\end{remark} 

Relatedly, consider the disjoint union $\Gamma_1\sqcup \Gamma_2$ of two webs. There is a natural homomorphism of graded $R$-modules 
\[
\functor(\Gamma_1)\otimes_R \functor(\Gamma_2) \lra \functor(\Gamma_1\sqcup \Gamma_2) 
\]
which is an isomorphism if either $\Gamma_1$ or $\Gamma_2$ is reducible. We don't know whether this map is an isomorphism for general $\Gamma_1,\Gamma_2$. Specializing $\Gamma_2=\Gamma_1^!$ to the reflection of $\Gamma$ in the plane, one can ask whether the $\Gamma_1$-tube foam (bent identity foam on $\Gamma_1$) from $\emptyset_1$ to $\Gamma_1\sqcup \Gamma_1^!$ is in the image of the homomorphism 
\[
\functor(\Gamma_1)\otimes_R \functor(\Gamma_1^!) \lra \functor(\Gamma_1\sqcup \Gamma_1^!). 
\]

An example of a non-reducible web is shown in Figure \ref{fig:ourgraph}. This web $\Gamma_{TST}$ is the graph of a solid called \emph{truncated square trapezohedron}. It has 16 vertices of type 112, two squared regions and 8 five-sided regions. $\Gamma_{TST}$ admits no 4-colorings with exactly four pigments. However, it can be colored with three and two pigments, as described in the second row of Figure~\ref{fig:ourgraph}. Observe that in the $SW$ corner of the figure each of the three pigments constitutes a single cycle, while in the $SE$ corner one of the pigments constitutes two cycles.

\begin{figure}
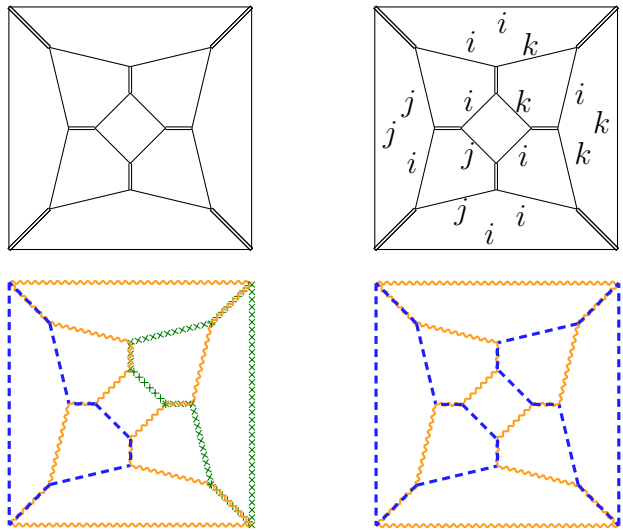

    \centering
\NB{\tikz[scale=0.9]{\input{\imagesfolder/4v_ourgraph}}}  \quad \quad \quad \,  \NB{\tikz[scale=0.9]{\input{\imagesfolder/4v_ourgraph_c1}}} \\ \vspace{0.3cm} \NB{\tikz[scale=0.9]{\input{\imagesfolder/4v_ourgraph_c2}}}  \quad \quad \quad\,  \NB{\tikz[scale=0.9]{\input{\imagesfolder/4v_ourgraph_c3}}}  
\caption{\small{Top left: the non-reducible web $\Gamma_{TST}$. Top right: a 3-coloring. Bottom left: $i$, $j$ and $k$ cycles of the coloring shown in yellow, blue and green, respectively. Bottom right: relabeling $k$ to $j$ (green to blue) results in a 2-coloring. Here $\{i,j,k\}\in \pigments$ and $i\neq j,k$. Any 4-coloring of $\Gamma_{TST}$ is obtained from these two colorings via permutations of colors and symmetries of the graph $\Gamma_{TST}$ in $\SS^2$. }} 
\label{fig:ourgraph}
\end{figure}

Any coloring $c$ of $\Gamma_{TST}$ gives rise to three cycles in the graph and exactly one of them is a Hamiltonian cycle in $\Gamma_{TST}$. In Figure~\ref{fig:ourgraph} examples, it's the $i$-cycle, shown in blue. 
A Kempe move from $c$ to $c'$ does not change this Hamiltonian cycle. One can check that, in this way, Kempe-equivalence classes of $\Gamma_{TST}$-colorings are in a bijection with Hamiltonian paths in this graph, of which there are four. 
Each of these classes contains 36 colorings, including 24 tricolorings and 12 bicolorings, for the total of 144 colorings, so that $t_4(\Gamma_{TST})=144$. The symmetric group $S_4$ acts transitively on each of these sets of 24 and 12 colorings, respectively. 

Web $\Gamma_{TST}$ is the non-reducible web with the smallest number of vertices that we were able to construct. 

\begin{remark}
We do not know the state space $\functor(\Gamma_{TST})$. It is possible to add new types of vertices to foams which are cones over $\Gamma_{TST}$ decorated by a Kempe-equivalence class of $\Gamma_{TST}$-colorings to define generalized $SL(4)$ webs and their evaluations, similar to the setup of \cite{KhovRob2}. In this enhanced setup, the state space $\functor(\Gamma_{TST})$ is isomorphic to the sum of four copies of the state space of the web \NB{\tikz[scale=0.35]{\begin{scope}
  \draw (0,1) -- (1,1) arc (90:-90:0.5) -- (0,0) arc (270:90:0.5);
  \draw[double] (0,0) -- (0,1);
  \draw[double] (1,0) -- (1,1);
\end{scope}}}, 
\[
\functor[\Gamma_{TST}] \ \cong \ 4 \, \functor[{\NB{\tikz[scale=0.35]{}}}]. \  
\] 
Colorings in a Kempe-equivalence class of $\Gamma_{TST}$-colorings are
in a bijection with colorings of the web
\NB{\tikz[scale=0.35]{}} with a natural match on cycles in the two colorings, leading to an isomorphism of state spaces.   
\end{remark}

\section{Tutte relations and localization}
\label{sec_localization}

\subsection{Periodic complexes}
\label{subsec_periodic}

Unfortunately functor $\eval$ does not properly categorify the Tutte
relations listed in Proposition~\ref{prop_tutte}. Instead there are
4-periodic complexes which are, in general, not exact lifting these relations. For instance, relation \eqref{eq_tutte_1} has a
categorical analogue given by the following 4-periodic chain
complex. Differentials in that complex are homogeneous, of degrees 1,
3, 1, 1, respectively, clockwise starting from the top differential.

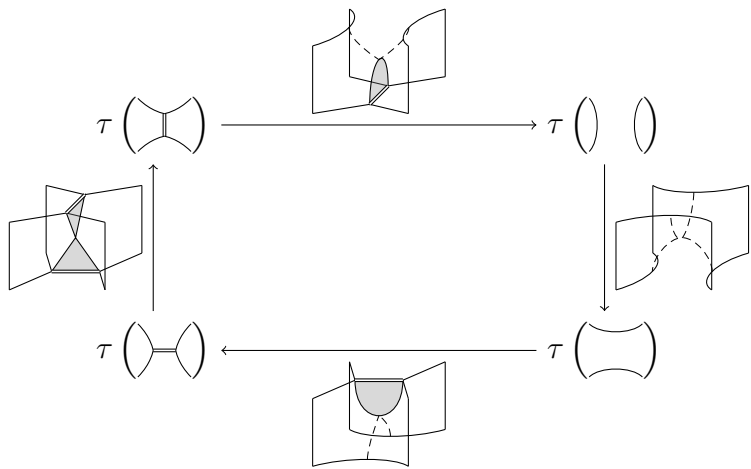
\begin{equation} \label{eq:4periodic-complex}
  \tikz[xscale =6, yscale =3, >=to]{
    \node (A) at (0,1) {$\functor[{\NB{\tikz[]{}}}]$};
    \node (B) at (1,1) {$\functor[{\NB{\tikz[]{}}}]$};
    \node (C) at (1,0) {$\functor[{\NB{\tikz[rotate=90]{}}}]$};
    \node (D) at (0,0) {$\functor[{\NB{\tikz[rotate=90]{}}}]$}; 
    \draw[white] (A) -- (B) node[pos=0.5, above] {\NB{\tikz[scale=0.9,black]{\begin{scope}
    \foreach \x in {45, 135, 225, 315}{
\begin{scope}[zxplane=-0.5, rotate=90]
    \coordinate (a\x) at (\x:0.5);
    \coordinate (A\x) at (\x:1);
  \end{scope}
  \begin{scope}[zxplane=0.5, rotate =90]
    \coordinate (b\x) at (\x:0.5);
    \coordinate (B\x) at (\x:1);
  \end{scope}
  \draw (A\x)--(B\x);}
\draw (A45) -- ($0.5*(a45)+0.5*(a135)$) --(A135);
\draw (A225) -- ($0.5*(a225)+0.5*(a315)$) --(A315);
\filldraw[fill = gray, fill opacity=0.3] ($0.5*(a45)+0.5*(a135)$)
.. controls +(0,0.7,0) and +(0, 0.7, 0) ..
($0.5*(a225)+0.5*(a315)$) coordinate[pos=0.5] (o) --cycle;
\draw[double] ($0.5*(a45)+0.5*(a135)$) -- ($0.5*(a225)+0.5*(a315)$); 
\draw (B45) .. controls (b45) and (b315) .. (B315) coordinate
[pos=0.5] (b1);
\draw (B225) .. controls (b225) and (b135) .. (B135) coordinate
[pos=0.5] (b2);
\draw[very thin,densely dashed] (b1) .. controls +(0,-0.2, 0) and
+(0,0,0) .. (o);
\draw[very thin,densely dashed] (b2) .. controls +(0,-0.2, 0) and
+(0,0,0) .. (o);
\end{scope}

}}};
    \draw[white] (B) -- (C) node[pos=0.5, right] {\NB{\tikz[scale=0.9,black]{\begin{scope}
    \foreach \x in {45, 135, 225, 315}{
\begin{scope}[zxplane=0.5, rotate=90]
    \coordinate (a\x) at (\x:0.5);
    \coordinate (A\x) at (\x:1);
  \end{scope}
  \begin{scope}[zxplane=-0.5, rotate =90]
    \coordinate (b\x) at (\x:0.5);
    \coordinate (B\x) at (\x:1);
  \end{scope}
  \draw (A\x)--(B\x);}
\draw (A45) .. controls (a45) and (a135) .. (A135) coordinate
[pos=0.5] (a1);
\draw (A225) .. controls (a225) and (a315) .. (A315) coordinate
[pos=0.5] (a2);
\draw (B45) .. controls (b45) and (b315) .. (B315) coordinate
[pos=0.5] (b1);
\draw (B225) .. controls (b225) and (b135) .. (B135) coordinate
[pos=0.5] (b2);
\coordinate (o) at (0,0,0);
\draw[very thin,densely dashed] (a1) .. controls +(0,-0.2,0) and
+(0,0,-0.2) .. (o)  .. controls +(0,0,0.20) and
+(0,-0.2,0) .. (a2);
\draw[very thin, densely dashed] (b1) .. controls +(0,0.2,0) and
+(0.2,0,0) .. (o)  .. controls +(-0.20,0,0) and
+(0,0.2,0) .. (b2);
\end{scope}

}}};
    \draw[white] (C) -- (D) node[pos=0.5, below] {\NB{\tikz[scale=0.9,black]{\begin{scope}
    \foreach \x in {45, 135, 225, 315}{
\begin{scope}[zxplane=0.5, rotate=0]
    \coordinate (a\x) at (\x:0.5);
    \coordinate (A\x) at (\x:1);
  \end{scope}
  \begin{scope}[zxplane=-0.5, rotate =0]
    \coordinate (b\x) at (\x:0.5);
    \coordinate (B\x) at (\x:1);
  \end{scope}
  \draw (A\x)--(B\x);}
\draw (A45) -- ($0.5*(a45)+0.5*(a135)$) --(A135);
\draw (A225) -- ($0.5*(a225)+0.5*(a315)$) --(A315);
\filldraw[fill = gray, fill opacity=0.3] ($0.5*(a45)+0.5*(a135)$)
.. controls +(0,-0.7,0) and +(0, -0.7, 0) ..
($0.5*(a225)+0.5*(a315)$) coordinate[pos=0.5] (o) --cycle;
\draw[double] ($0.5*(a45)+0.5*(a135)$) -- ($0.5*(a225)+0.5*(a315)$); 
\draw (B45) .. controls (b45) and (b315) .. (B315) coordinate
[pos=0.5] (b1);
\draw (B225) .. controls (b225) and (b135) .. (B135) coordinate
[pos=0.5] (b2);
\draw[very thin,densely dashed] (b1) .. controls +(0,0.2, 0) and
+(0,0,0) .. (o);
\draw[very thin,densely dashed] (b2) .. controls +(0,0.2, 0) and
+(0,0,0) .. (o);
\end{scope}

}}};
    \draw[white] (D) -- (A) node[pos=0.5, left]
    {\NB{\tikz[scale=0.9,black]{\begin{scope}
    \foreach \x in {45, 135, 225, 315}{
\begin{scope}[zxplane=-0.5, rotate=0]
    \coordinate (a\x) at (\x:0.5);
    \coordinate (A\x) at (\x:1);
  \end{scope}
  \begin{scope}[zxplane=0.5, rotate =0]
    \coordinate (b\x) at (\x:0.5);
    \coordinate (B\x) at (\x:1);
  \end{scope}
  \draw (A\x)--(B\x);}
\draw (A45) -- ($0.5*(a45)+0.5*(a135)$) --(A135);
\draw (A225) -- ($0.5*(a225)+0.5*(a315)$) --(A315);
\draw (B45) -- ($0.5*(b45)+0.5*(b315)$) --(B315);
\draw (B225) -- ($0.5*(b225)+0.5*(b135)$) --(B135);
\coordinate (o) at (0,0,0);
\filldraw[fill = gray, fill opacity=0.3] ($0.5*(a45)+0.5*(a135)$) --(o)--
($0.5*(a225)+0.5*(a315)$) -- cycle;
\filldraw[fill = gray, fill opacity=0.3] ($0.5*(b45)+0.5*(b315)$) --(o)--
($0.5*(b225)+0.5*(b135)$) -- cycle;
\draw[double] ($0.5*(a45)+0.5*(a135)$) -- ($0.5*(a225)+0.5*(a315)$);
\draw[double] ($0.5*(b45)+0.5*(b315)$) -- ($0.5*(b225)+0.5*(b135)$); 
\end{scope}

}}};
    \draw[-to] (A) --(B);
    \draw[-to] (B) --(C);
    \draw[-to] (C) --(D);
    \draw[-to] (D) --(A);
  }
\end{equation}

In this square the composition of two consecutive arrows is $0$ since
the foams obtained by stacking two consecutive morphisms do not
admit any coloring. This periodic chain complex is in general not
null-homotopic.

Relations \eqref{eq_tutte_2} and \eqref{eq_tutte_3} have the following
counterparts:
\begin{equation*} \label{eq:4periodic-complex2}
  \NB{\tikz[xscale =3, yscale =1.5, >=to]{
    \node (A) at (0,1) {$\functor[{\NB{\tikz[rotate=180]{\input{\imagesfolder/12v12}}}}]$};
    \node (B) at (1,1) {$\functor[{\NB{\tikz[rotate=180]{}}}]$};
    \node (C) at (1,0) {$\functor[{\NB{\tikz[rotate=180]{}}}]$};
    \node (D) at (0,0) {$\functor[{\NB{\tikz[rotate=-90]{\input{\imagesfolder/11v22}}}}]$}; 
\draw[-to] (A) --(B);
    \draw[-to] (B) --(C);
    \draw[-to] (C) --(D);
    \draw[-to] (D) --(A);
  }}\qquad \text{and} \qquad 
\NB{\tikz[xscale =3, yscale =1.5, >=to]{
    \node (A) at (0,1) {$\functor[{\NB{\tikz[rotate=90]{\input{\imagesfolder/12defv21}}}}]$};
    \node (B) at (1,1) {$\functor[{\NB{\tikz[xscale=-1]{\input{\imagesfolder/12defv21}}}}]$};
    \node (C) at (1,0) {$\functor[{\NB{\tikz[rotate=270]{\input{\imagesfolder/12defv21}}}}]$};
    \node (D) at (0,0) {$\functor[{\NB{\tikz[yscale=-1]{\input{\imagesfolder/12defv21}}}}]$}; 
\draw[-to] (A) --(B);
    \draw[-to] (B) --(C);
    \draw[-to] (C) --(D);
    \draw[-to] (D) --(A);
  }}
\end{equation*}

\subsection{Localization}\label{subsec_localization} 

Given a grading-preserving homomorphism $\psi:\ourring\lra S$ of commutative rings, define a foam evaluation associated to it by 
\begin{equation*}
    \functor_{\psi}(F)\ :=\ \psi(\functor(F)) \ \in S. 
\end{equation*}
Evaluation $\functor_{\psi}$ gives rise to state spaces $\functor_{\psi}(\Gamma)$ of webs and  graded $S$-module homomorphisms between state spaces $\tau_{\psi}(F):\partial_0 F{\lra}\partial_1 F$ assigned to foams $F$ with boundary. These state spaces and homomorphisms between them form a lax monoidal functor 
\begin{equation*}
    \functor_{\psi} \ : \ \catFoam \lra S\mathsf{-gmod}. 
\end{equation*}
The natural grading-preserving homomorphism 
\[\functor(\Gamma)\otimes_R S\lra \functor_{\psi}(\Gamma)
\]
of $S$-modules is an isomorphism if $\functor(\Gamma)$ is a free graded $R$-module. In particular, it is an isomorphism for all reducible webs. 

Theorem~\ref{thm_dir_sum_rels}  holds for functor $\functor_{\psi}$ in place of $\functor$. 

Assume that the commutative ring $S$ is zero in negative degrees and a characteristic two field $S_0\supset \mathbf{F}_2$ in degree $0$, so that $S=S_0\oplus S_1\oplus S_2\oplus \dots$ Then $S$ is a graded local ring and any finitely-generated graded projective module $P$ over it has a well-defined graded rank $\rank_{q,S}(P)\in \Z[q,q^{-1}]$. 

\begin{proposition}
    If $\Gamma$ is a reducible web,
    \[
     \rank_{q,S}(\tau_{\psi}\Gamma) \ = \ \rank_q(\tau(\Gamma))
    \]
    for any homomorphism $\psi:R\lra S$ into a graded local ring as above. 
\end{proposition}

Assume that some element $x$ of $S$ of positive degree $d>0$ is invertible. Then a free rank one $S$-module $S$ is isomorphic to the shifted module $S\{d\}$, and one can at most hope to define $q$-degree as an element of the quotient ring $\Z[q]/(q^d-1)$ of $\Z[q,q^{-1}]$. Note that, at least, for any commutative ring $S$, any finitely-generated free module $P$ over $S$ has a well-defined rank given by picking a maximal ideal $\mathfrak{m}\subset S$ and defining 
$\rank_S(P)=\dim_{S/\mathfrak{m}}(P/\mathfrak{m}P).$ Furthermore, if $S$ is a noetherian integral domain of finite homological dimension,  
this construction can be extended to a finitely-generated $S$-module $M$ by picking a finite length resolution $P^{\ast}(M)$ of $M$ with finitely-generated terms, tensoring all terms with the fraction field $Q(S)$ and taking the Euler characteristic of the resulting complex: 
\[
M\rightsquigarrow \chi(P^{\ast}(M)\otimes_S Q(S)). 
\]
This allows to define the corresponding quantum invariant 
\[
\chi_{\psi}(\Gamma) \ := \chi(P^{\ast}(\functor_{\psi}(\Gamma))\otimes_S Q(S)) 
\]
of webs $\Gamma$ 
for ring homomorphisms $\psi:R\lra S$ into noetherian integral domains $S$ of finite homological dimension, where $S$ does not have to be graded. If $\Gamma$ is reducible, $\chi_{\psi}(\Gamma)=t_4(\Gamma)$ for any such $\psi$. 

We now inspect how the theory behaves when inverting
\begin{equation*}
\discriminant \ = \ \prod_{1\leq i < j \leq 4} (X_i+X_j)\ = \ E_3E_2E_1+E_3^2+E_4E_1^2 \in
\ourring,
\end{equation*}
where $\delta$ is the square root of the discriminant of the polynomial $x^4+E_1 x^3 + E_2 x^2+E_3 x + E_4$ (in characteristic two square root of the discriminant is a symmetric polynomial), and $E_i$ is the $i$-th elementary symmetric function of variables $X_1,X_2,X_3,X_4$. 

Denote $\localring=\ourring[\discriminant^{-1}]$ and consider the
homomorphism $\psi_{\delta}:R\lra \localring$ as above (taking
$S=\localring$). The ring $\localring$ is a noetherian integral domain of homological dimension three. As we shall
see, over this ring we can lift Tutte-like relations of Proposition~\ref{prop_tutte} to periodic exact sequences. However, 
$\localring$ has homogeneous invertible elements of degree $2$ (for
instance, elements $X_i+X_j$), so that information about the grading
on state spaces is lost upon this localization.

Using $\discriminant^{-1}$ offers more flexibility and allows to
define new decoration called \emph{hollow dots} which will be depicted
by $\circledast$ on thin facets and by $\circ$ on thick facets.

The hollow dot $\circledast$ is defined by:
\[
  \NB{\tikz[]{\begin{scope}
  \draw (0,0) rectangle (1,1) node[pos=0.5,scale=0.7] {$\circledast$};
\end{scope}}} = \frac{1}{\discriminant}\left(
    (E_3 +
    E_2E_1)\ \NB{\tikz[]{\begin{scope}
  \draw (0,0) rectangle (1,1);
\end{scope}}} \ +
    E_1^2\ \NB{\tikz[]{\begin{scope}
  \draw (0,0) rectangle (1,1) node[pos=0.5] {$\bullet$};
\end{scope}}} + E_1\ \NB{\tikz[]{\begin{scope}
  \draw (0,0) rectangle (1,1) node[pos=0.5] {$\bullet^2$};
\end{scope}}}\right).
\]
More concretely, a hollow dot $\circledast$ on a thin facet colored by
$i$ contributes $\frac{1}{(X_i+X_j)(X_i+X_k)(X_i+X_{\ell})}$ to the evaluation formula.

The hollow dot $\circ$ is defined by:
\begin{align}
  \NB{\tikz[scale=1.4]{\begin{scope}
  \filldraw[fill opacity=0.3, fill =gray] (0,0) rectangle (1,2)
  coordinate[pos=0.5] (a);
  \node at (0.5, 1) {$\circ$};
\end{scope}}}&=\frac{1}{\discriminant}\left(
\NB{\tikz[font=\tiny, scale=1.4]{\begin{scope}
  \filldraw[fill opacity=0.3, fill =gray, double] (0,0) rectangle (1,2);;
  coordinate[pos=0.5] (a);
  \draw[thick, orange] (0.5, 1.5) circle (0.4);
  \node[scale=0.9] at (0.5, 0.5) {$\bullet e_2^2$};
  \node[scale=0.9] at (0.5, 1.5) {$\bullet e_1$};
\end{scope}}}+
\NB{\tikz[font=\tiny, scale=1.4]{\begin{scope}
  \filldraw[fill opacity=0.3, fill =gray, double] (0,0) rectangle (1,2);;
  coordinate[pos=0.5] (a);
  \draw[thick, orange] (0.5, 1.5) circle (0.4);
  \node[scale=0.9] at (0.5, 0.5) {$\bullet e_2e_1$};
  \node[scale=0.9] at (0.5, 1.5) {$\bullet e_1^2$};
\end{scope}}}+
\NB{\tikz[font=\tiny, scale=1.4]{\begin{scope}
  \filldraw[fill opacity=0.3, fill =gray, double] (0,0) rectangle (1,2);;
  coordinate[pos=0.5] (a);
  \draw[thick, orange] (0.5, 1.5) circle (0.4);
  \node[scale=0.9] at (0.5, 0.5) {$\bullet e_1^2$};
  \node[scale=0.9] at (0.5, 1.5) {$\bullet e_2e_1$};
\end{scope}}}+
\NB{\tikz[font=\tiny, scale=1.4]{\begin{scope}
  \filldraw[fill opacity=0.3, fill =gray, double] (0,0) rectangle (1,2);;
  coordinate[pos=0.5] (a);
  \draw[thick, orange] (0.5, 1.5) circle (0.4);
  \node[scale=0.9] at (0.5, 0.5) {$\bullet e_2$};
  \node[scale=0.9] at (0.5, 1.5) {$\bullet e_1^3$};
\end{scope}}}+
\NB{\tikz[font=\tiny, scale=1.4]{\begin{scope}
  \filldraw[fill opacity=0.3, fill =gray, double] (0,0) rectangle (1,2);;
  coordinate[pos=0.5] (a);
  \draw[thick, orange] (0.5, 1.5) circle (0.4);
  \node[scale=0.9] at (0.5, 0.5) {$\bullet e_1$};
  \node[scale=0.9] at (0.5, 1.5) {$\bullet e_1^2e_1$};
\end{scope}}}+
\NB{\tikz[font=\tiny, scale=1.4]{\begin{scope}
  \filldraw[fill opacity=0.3, fill =gray, double] (0,0) rectangle (1,2);;
  coordinate[pos=0.5] (a);
  \draw[thick, orange] (0.5, 1.5) circle (0.4);
  \node[scale=0.9] at (0.5, 1.5) {$\bullet e_1e_2^2$};
\end{scope}}}
  \right)
\end{align}

In other words, an hollow dot on a thick
facet colored by $\{i,j\}$, contributes, as a decoration, by 
$\frac{1}{X_i+X_j}$ to the evaluation formula.  It is also
characterized by:
\begin{equation}
  \label{eq:inverteddot}
  \NB{\tikz[]{\begin{scope}
  \filldraw[fill opacity=0.3, fill =gray] (0,0) rectangle (1,1)
  coordinate[pos=0.5] (a);
  \node at (0.5, 0.5) {$\circ\ \ \bullet$};
\end{scope}}} =   \NB{\tikz[]{\begin{scope}
  \filldraw[fill opacity=0.3, fill =gray] (0,0) rectangle (1,1)
  coordinate[pos=0.5] (a);
\end{scope}}} 
\end{equation}
and can be written as the inverse of $\bullet$. 

With this new decoration at hand, one can show that the 4-periodic complex
\eqref{eq:4periodic-complex} (with $\functord$ instead of $\functor$)
is exact. Indeed we have the following sections

\begin{equation*} \label{eq:4periodic-complex-sections}
  \tikz[xscale =6, yscale =3, >=to]{
    \node (A) at (0,1) {$\functord[{\NB{\tikz[]{}}}]$};
    \node (B) at (1,1) {$\functord[{\NB{\tikz[]{}}}]$};
    \node (C) at (1,0) {$\functord[{\NB{\tikz[rotate=90]{}}}]$};
    \node (D) at (0,0) {$\functord[{\NB{\tikz[rotate=90]{}}}]$}; 
    \draw[white] (A) -- (B) coordinate[yshift=1cm, pos=0.5] (ab);
    \node at (ab) {\NB{\tikz[scale=1,black]{\begin{scope}
    \foreach \x in {45, 135, 225, 315}{
\begin{scope}[zxplane=0.5, rotate=90]
    \coordinate (a\x) at (\x:0.5);
    \coordinate (A\x) at (\x:1);
  \end{scope}
  \begin{scope}[zxplane=-0.5, rotate =90]
    \coordinate (b\x) at (\x:0.5);
    \coordinate (B\x) at (\x:1);
  \end{scope}
  \draw (A\x)--(B\x);}
\node at (0,0.3, 0) {$\circ$};
\draw (A45) -- ($0.5*(a45)+0.5*(a135)$) --(A135);
\draw (A225) -- ($0.5*(a225)+0.5*(a315)$) --(A315);
\filldraw[fill = gray, fill opacity=0.3] ($0.5*(a45)+0.5*(a135)$)
.. controls +(0,-0.7,0) and +(0, -0.7, 0) ..
($0.5*(a225)+0.5*(a315)$) coordinate[pos=0.5] (o) --cycle;
\draw[double] ($0.5*(a45)+0.5*(a135)$) -- ($0.5*(a225)+0.5*(a315)$); 
\draw (B45) .. controls (b45) and (b315) .. (B315) coordinate
[pos=0.5] (b1);
\draw (B225) .. controls (b225) and (b135) .. (B135) coordinate
[pos=0.5] (b2);
\draw[very thin,densely dashed] (b1) .. controls +(0,0.2, 0) and
+(0,0,0) .. (o);
\draw[very thin,densely dashed] (b2) .. controls +(0,0.2, 0) and
+(0,0,0) .. (o);
\end{scope}

}}};
    \draw[white] (B) -- (C) coordinate[xshift=1.5cm, pos=0.5] (bc);
    \node at (bc) {\NB{\tikz[scale=1,black]{\begin{scope}
    \foreach \x in {45, 135, 225, 315}{
\begin{scope}[zxplane=0.5, rotate=0]
    \coordinate (a\x) at (\x:0.5);
    \coordinate (A\x) at (\x:1);
  \end{scope}
  \begin{scope}[zxplane=-0.5, rotate =0]
    \coordinate (b\x) at (\x:0.5);
    \coordinate (B\x) at (\x:1);
  \end{scope}
  \draw (A\x)--(B\x);}
\node at ($0.7*(A45)+0.3*(B135)$) {$\circledast$};
\draw (A45) .. controls (a45) and (a135) .. (A135) coordinate
[pos=0.5] (a1);
\draw (A225) .. controls (a225) and (a315) .. (A315) coordinate
[pos=0.5] (a2);
\draw (B45) .. controls (b45) and (b315) .. (B315) coordinate
[pos=0.5] (b1);
\draw (B225) .. controls (b225) and (b135) .. (B135) coordinate
[pos=0.5] (b2);
\coordinate (o) at (0,0,0);
\draw[very thin, densely dashed] (a1) .. controls +(0,-0.2,0) and
+(0.2,0,0) .. (o)  .. controls +(-0.2,0,0) and
+(0,-0.2,0) .. (a2);
\draw[very thin, densely dashed] (b1) .. controls +(0,0.2,0) and
+(0,0,0.2) .. (o)  .. controls +(0,0,-0.2) and
+(0,0.2,0) .. (b2);
\end{scope}}}};
    \draw[white] (C) -- (D) coordinate[yshift=-1cm, pos=0.5] (cd);
    \node at (cd) {\NB{\tikz[scale=1,black]{\begin{scope}
    \foreach \x in {45, 135, 225, 315}{
\begin{scope}[zxplane=-0.5, rotate=0]
    \coordinate (a\x) at (\x:0.5);
    \coordinate (A\x) at (\x:1);
  \end{scope}
  \begin{scope}[zxplane=0.5, rotate =0]
    \coordinate (b\x) at (\x:0.5);
    \coordinate (B\x) at (\x:1);
  \end{scope}
  \draw (A\x)--(B\x);}
\node at (0,-0.3, 0) {$\circ$};
\draw (A45) -- ($0.5*(a45)+0.5*(a135)$) --(A135);
\draw (A225) -- ($0.5*(a225)+0.5*(a315)$) --(A315);
\filldraw[fill = gray, fill opacity=0.3] ($0.5*(a45)+0.5*(a135)$)
.. controls +(0,0.7,0) and +(0, 0.7, 0) ..
($0.5*(a225)+0.5*(a315)$) coordinate[pos=0.5] (o) --cycle;
\draw[double] ($0.5*(a45)+0.5*(a135)$) -- ($0.5*(a225)+0.5*(a315)$); 
\draw (B45) .. controls (b45) and (b315) .. (B315) coordinate
[pos=0.5] (b1);
\draw (B225) .. controls (b225) and (b135) .. (B135) coordinate
[pos=0.5] (b2);
\draw[very thin,densely dashed] (b1) .. controls +(0,-0.2, 0) and
+(0,0,0) .. (o);
\draw[very thin,densely dashed] (b2) .. controls +(0,-0.2, 0) and
+(0,0,0) .. (o);
\end{scope}

}}};
    \draw[white] (D) -- (A) coordinate [xshift=-1.5cm, pos=0.5] (da);
    \node at (da) {\NB{\tikz[scale=1,black]{\begin{scope}
    \foreach \x in {45, 135, 225, 315}{
\begin{scope}[zxplane=0.5, rotate=0]
    \coordinate (a\x) at (\x:0.5);
    \coordinate (A\x) at (\x:1);
  \end{scope}
  \begin{scope}[zxplane=-0.5, rotate =0]
    \coordinate (b\x) at (\x:0.5);
    \coordinate (B\x) at (\x:1);
  \end{scope}
  \draw (A\x)--(B\x);}
\node at (0,0.3,0) {$\circ$};
\draw (A45) -- ($0.5*(a45)+0.5*(a135)$) --(A135);
\draw (A225) -- ($0.5*(a225)+0.5*(a315)$) --(A315);
\draw (B45) -- ($0.5*(b45)+0.5*(b315)$) --(B315);
\draw (B225) -- ($0.5*(b225)+0.5*(b135)$) --(B135);
\coordinate (o) at (0,0,0);
\filldraw[fill = gray, fill opacity=0.3] ($0.5*(a45)+0.5*(a135)$) --(o)--
($0.5*(a225)+0.5*(a315)$) -- cycle;
\filldraw[fill = gray, fill opacity=0.3] ($0.5*(b45)+0.5*(b315)$) --(o)--
($0.5*(b225)+0.5*(b135)$) -- cycle;
\draw[double] ($0.5*(a45)+0.5*(a135)$) -- ($0.5*(a225)+0.5*(a315)$);
\draw[double] ($0.5*(b45)+0.5*(b315)$) -- ($0.5*(b225)+0.5*(b135)$); 
\end{scope}

}}};
    \draw[densely dashed, to-] (A) --(B);
    \draw[densely dashed, to-] (B) --(C);
    \draw[densely dashed, to-] (C) --(D);
    \draw[densely dashed, to-] (D) --(A);
    \path[-to] (B) edge[bend left] (A);
    \path[-to] (C) edge[bend left] (B);
    \path[-to] (D) edge[bend left] (C);
    \path[-to] (A) edge[bend left] (D);
  }
\end{equation*}

which provide the following isomorphism:
\begin{equation*}  \label{eq_delta_1_0}
 \functord[\NB{\tikz[]{}}] \oplus \functord[{\NB{\tikz[rotate=90]{}}}]  \cong 
     \functord[{\NB{\tikz[rotate=90]{}}}] \oplus \functord[\NB{\tikz[]{}}]. 
\end{equation*}

Similar foams with hollow dots, give analogous results for other
periodic complexes, so that we obtain:

\begin{proposition}\label{prop_functd} 
  The functor $\functord$ satisfies the following relations (compare to Proposition \ref{prop_tutte}):
  \begin{align}
  \label{eq_delta_1}
 \functord[\NB{\tikz[]{}}] \oplus \functord[{\NB{\tikz[rotate=90]{}}}] & \cong 
     \functord[{\NB{\tikz[rotate=90]{}}}] \oplus \functord[\NB{\tikz[]{}}], \\
     \label{eq_delta_2}
     \functord[{\NB{\tikz[rotate=180]{\input{\imagesfolder/12v12}}}}] \oplus
    \functord[{\NB{\tikz[rotate=180]{}}}] & \cong 
    \functord[{\NB{\tikz[rotate=180]{}}}] \oplus \functord[{\NB{\tikz[rotate=-90]{\input{\imagesfolder/11v22}}}}], \\
    \label{eq_delta_3}
     \functord[{\NB{\tikz[rotate=90]{\input{\imagesfolder/12defv21}}}}] \oplus \functord[{\NB{\tikz[rotate=270]{\input{\imagesfolder/12defv21}}}}] & \cong 
    \functord[{\NB{\tikz[xscale=-1]{\input{\imagesfolder/12defv21}}}}]\oplus\functord[{\NB{\tikz[yscale=-1]{\input{\imagesfolder/12defv21}}}}].\end{align}
\end{proposition}

Denote by $t_{q=1}(\Gamma)\in \Z$ the specialization of the Laurent polynomial $t_q(\Gamma)\in \Z[q,q^{-1}]$ to $q=1$. 

\begin{theorem}
    For any web $\Gamma$ the state space $\functord(\Gamma)$ is a free $R_{\delta}$-module of rank $t_4(\Gamma)$ and 
    \[
    \rank_{R_\delta}(\functord(\Gamma)) = t_4(\Gamma)=t_{q=1}(\Gamma).
    \]
\end{theorem}
\begin{proof}
  The first equality follows from Propositions~\ref{prop:t4computable} and
  \ref{prop_functd} and Theorem~\ref{thm_dir_sum_rels}. Proposition~\ref{prop_functd} and Theorem~\ref{thm_dir_sum_rels} imply that
  $\rank_{R_\delta}(\functord(\ \cdot\ ))$ satisfies the same
  relations as $t_4$. Propositions~\ref{prop:t4computable} implies that these relations
  characterize $t_4$ completely.
\end{proof}

Suppose now that $S\cong \kk[E_0]$, where $\deg(E_0)=2s,$ for some $s\in \{1,2,3,4\}$. 
Pick a graded ring homomorphism $\psi:R\lra \kk[E_0]$ sending generators $E_i$, $1 \le i\le 4$ to scaled powers of $E_0$ or to $0$, depending on the degree. For instance, if $s=1$, there are homomorphisms
\[ \psi(E_1)=\mu_1 E_0,\ \psi(E_2)=\mu_2 E_0^2,\ \psi(E_3)=\mu_3 E_0^3, \ \psi(E_4)=\mu_4 E_4, \ \  \mu_i\in \kk,  \]
with not all $\mu_i=0$.

\begin{proposition} If $\psi(\delta)\not=0$ then for any web $\Gamma$ 
the state space $\functor_{\psi}(\Gamma)$ is a free graded $\kk[E_0]$-module of (ungraded) rank $t_4(\Gamma).$
\end{proposition}

The proof is identical to that of Propositions~4.5 and 4.18 in~\cite{KhoRozI}. $\square$

\vspace{0.07in} 

We do not know how to compute the graded rank of $\functor_{\psi}(\Gamma)$ when $\Gamma$ is not reducible and whether for some nonreducible graphs the rank may depend on $\psi$.  

If $\psi(\delta)=0$, at most we can derive is the inequality $\rank(\functor_{\psi}(\Gamma))\le t_4(\Gamma)$ (compare with~\cite[Corollary 2.11]{Boozer1}). 

\vspace{0.07in}

Likewise, assume given a homomorphism $\psi:R\lra \kk$ into a field $\kk$ and ignore the grading of $R$. 
\begin{itemize}
    \item 
If $\psi(\delta)\not=0\in\kk$ then the theory can be fully understood, via Proposition~\ref{prop_functd}, with $\kk$ in place of $R_{\delta}$. $\kk$-vector space $\functor_{\psi}(\Gamma)$ is not graded and $\dim(\functor_{\psi}(\Gamma))=t_4(\Gamma).$ 
\item If $\psi(\delta)=0\in\kk$ then $\functor_{\psi}(\Gamma)$ is a finite-dimensional graded vector space and $\dim(\functor_{\psi}(\Gamma))\le t_4(\Gamma)$, with equality at least for reducible $\Gamma$ (c.f.~\cite[Corollary 2.11]{Boozer1}). 
\end{itemize}

\bibliographystyle{alphaurl}
\bibliography{ref.bib}

\begin{thebibliography}{BHMV95}

\bibitem[Aga22]{Aganagic}
M.~Aganagic.
\newblock Homological knot invariants from mirror symmetry, 2022.
\newblock \href {http://arxiv.org/abs/2207.14104} {\path{arXiv:2207.14104}}.

\bibitem[AN23]{AnnoNandakumar}
R.~Anno and Vinoth N.
\newblock Exotic $t$-structures for two-block {S}pringer fibres.
\newblock {\em Trans. Amer. Math. Soc.}, 376(3):1523--1552, 2023.
\newblock \href {http://arxiv.org/abs/1602.00768} {\path{arXiv:1602.00768}},
  \href {https://doi.org/10.1090/tran/8765} {\path{doi:10.1090/tran/8765}}.

\bibitem[AS19]{AbouzaidSmith}
M.~Abouzaid and I.~Smith.
\newblock Khovanov homology from {F}loer cohomology.
\newblock {\em J. Amer. Math. Soc.}, 32(1):1--79, 2019.
\newblock \href {http://arxiv.org/abs/1504.01230} {\path{arXiv:1504.01230}},
  \href {https://doi.org/10.1090/jams/902} {\path{doi:10.1090/jams/902}}.

\bibitem[BFK99]{BernsteinFrenkelKhovanov}
J.~Bernstein, I.~Frenkel, and M.~Khovanov.
\newblock A categorification of the {T}emperley-{L}ieb algebra and {S}chur
  quotients of {$U(\germ{sl}_2)$} via projective and {Z}uckerman functors.
\newblock {\em Selecta Math. (N.S.)}, 5(2):199--241, 1999.
\newblock \href {http://arxiv.org/abs/math/0002087}
  {\path{arXiv:math/0002087}}, \href {https://doi.org/10.1007/s000290050047}
  {\path{doi:10.1007/s000290050047}}.

\bibitem[BHMV95]{BHMV}
C.~Blanchet, N.~Habegger, G.~Masbaum, and P.~Vogel.
\newblock Topological quantum field theories derived from the {K}auffman
  bracket.
\newblock {\em Topology}, 34(4):883--927, 1995.
\newblock \href {https://doi.org/10.1016/0040-9383(94)00051-4}
  {\path{doi:10.1016/0040-9383(94)00051-4}}.

\bibitem[Boo21]{Boozer1}
D.~Boozer.
\newblock Computer bounds for {K}ronheimer--{M}rowka foam evaluation.
\newblock {\em Experimental Mathematics}, 0(0):1--16, 2021.
\newblock \href {http://arxiv.org/abs/1908.07133} {\path{arXiv:1908.07133}},
  \href {https://doi.org/10.1080/10586458.2021.1982078}
  {\path{doi:10.1080/10586458.2021.1982078}}.

\bibitem[Boo23]{Boozer2}
D.~Boozer.
\newblock The combinatorial and gauge-theoretic foam evaluation functors are
  not the same, 2023.
\newblock \href {http://arxiv.org/abs/2304.07659} {\path{arXiv:2304.07659}}.

\bibitem[CEKS15]{CEKS}
M.~Chudnovsky, K.~Edwards, K.~Kawarabayashi, and P.~Seymour.
\newblock Edge-colouring seven-regular planar graphs.
\newblock {\em J. Combin. Theory Ser. B}, 115:276--302, 2015.
\newblock \href {http://arxiv.org/abs/1210.7349} {\path{arXiv:1210.7349}},
  \href {https://doi.org/10.1016/j.jctb.2014.11.005}
  {\path{doi:10.1016/j.jctb.2014.11.005}}.

\bibitem[CK08]{CautisKamnitzer}
S.~Cautis and J.~Kamnitzer.
\newblock Knot homology via derived categories of coherent sheaves. {II}.
  {$\germ{sl}_m$} case.
\newblock {\em Invent. Math.}, 174(1):165--232, 2008.
\newblock \href {http://arxiv.org/abs/0710.3216} {\path{arXiv:0710.3216}},
  \href {https://doi.org/10.1007/s00222-008-0138-6}
  {\path{doi:10.1007/s00222-008-0138-6}}.

\bibitem[GSV05]{GukovScwarzVafa}
S.~Gukov, A.~Schwarz, and C.~Vafa.
\newblock Khovanov-{R}ozansky homology and topological strings.
\newblock {\em Lett. Math. Phys.}, 74(1):53--74, 2005.
\newblock \href {http://arxiv.org/abs/hep-th/0412243}
  {\path{arXiv:hep-th/0412243}}, \href
  {https://doi.org/10.1007/s11005-005-0008-8}
  {\path{doi:10.1007/s11005-005-0008-8}}.

\bibitem[Gue14]{Guenin}
B.~Guenin.
\newblock Packing ${T}$-joins and edge-colouring in planar graphs.
\newblock To appear in \emph{Math.{} Oper.{} Res.{}}, 2014.

\bibitem[Kho04]{Kho04}
M.~Khovanov.
\newblock sl(3) link homology.
\newblock {\em Algebr. Geom. Topol.}, 4:1045--1081, 2004.
\newblock \href {http://arxiv.org/abs/math/0304375v2}
  {\path{arXiv:math/0304375v2}}, \href
  {https://doi.org/10.2140/agt.2004.4.1045}
  {\path{doi:10.2140/agt.2004.4.1045}}.

\bibitem[Kho20]{KhovUniversal}
M.~Khovanov.
\newblock Universal construction of topological theories in two dimensions,
  2020.
\newblock \href {http://arxiv.org/abs/2007.03361} {\path{arXiv:2007.03361}}.

\bibitem[KK20]{KhoKit}
M.~Khovanov and N.~Kitchloo.
\newblock A deformation of {R}obert--{W}agner foam evaluation and link
  homology, 2020.
\newblock To appear in \textit{Contemp. Math.}
\newblock \href {http://arxiv.org/abs/2004.14197} {\path{arXiv:2004.14197}}.

\bibitem[KM16]{KM3}
P.~B. Kronheimer and T.~S. Mrowka.
\newblock Exact triangles for {$SO(3)$} instanton homology of webs.
\newblock {\em J. Topol.}, 9(3):774--796, 2016.
\newblock \href {http://arxiv.org/abs/1508.07207} {\path{arXiv:1508.07207}},
  \href {https://doi.org/10.1112/jtopol/jtw010}
  {\path{doi:10.1112/jtopol/jtw010}}.

\bibitem[KM19a]{KM2}
P.~B. Kronheimer and T.~S. Mrowka.
\newblock A deformation of instanton homology for webs.
\newblock {\em Geom. Topol.}, 23(3):1491--1547, 2019.
\newblock \href {http://arxiv.org/abs/1710.05002} {\path{arXiv:1710.05002}},
  \href {https://doi.org/10.2140/gt.2019.23.1491}
  {\path{doi:10.2140/gt.2019.23.1491}}.

\bibitem[KM19b]{KM1}
P.~B. Kronheimer and T.~S. Mrowka.
\newblock Tait colorings, and an instanton homology for webs and foams.
\newblock {\em J. Eur. Math. Soc. (JEMS)}, 21(1):55--119, 2019.
\newblock \href {http://arxiv.org/abs/1508.07205} {\path{arXiv:1508.07205}},
  \href {https://doi.org/10.4171/JEMS/831} {\path{doi:10.4171/JEMS/831}}.

\bibitem[KR08]{KhoRozI}
M.~Khovanov and L.~Rozansky.
\newblock Matrix factorizations and link homology.
\newblock {\em Fund. Math.}, 199(1):1--91, 2008.
\newblock \href {http://arxiv.org/abs/math/0401268}
  {\path{arXiv:math/0401268}}, \href {https://doi.org/10.4064/fm199-1-1}
  {\path{doi:10.4064/fm199-1-1}}.

\bibitem[KR21]{KhovRob1}
M.~Khovanov and L.-H. Robert.
\newblock Foam evaluation and {K}ronheimer-{M}rowka theories.
\newblock {\em Adv. Math.}, 376:Paper No. 107433, 59, 2021.
\newblock \href {http://arxiv.org/abs/1808.09662} {\path{arXiv:1808.09662}},
  \href {https://doi.org/10.1016/j.aim.2020.107433}
  {\path{doi:10.1016/j.aim.2020.107433}}.

\bibitem[KR22]{KhovRob2}
M.~Khovanov and L.-H. Robert.
\newblock Conical {$SL(3)$} foams.
\newblock {\em J. Comb. Algebra}, 6(1-2):79--108, 2022.
\newblock \href {http://arxiv.org/abs/2011.11077} {\path{arXiv:2011.11077}},
  \href {https://doi.org/10.4171/jca/61} {\path{doi:10.4171/jca/61}}.

\bibitem[Man07]{Manolescu}
C.~Manolescu.
\newblock Link homology theories from symplectic geometry.
\newblock {\em Adv. Math.}, 211(1):363--416, 2007.
\newblock \href {http://arxiv.org/abs/math/0601629}
  {\path{arXiv:math/0601629}}, \href
  {https://doi.org/10.1016/j.aim.2006.09.007}
  {\path{doi:10.1016/j.aim.2006.09.007}}.

\bibitem[Mru23]{Mrudul}
M.~T. Mrudul.
\newblock {\em Webs and Foams of Simple Lie Algebras}.
\newblock PhD thesis, Columbia University, 2023.

\bibitem[MSV09]{MacStoVaz}
M.~Mackaay, M.~Sto\v{s}i\'{c}, and Pedro Vaz.
\newblock {$\mathfrak{sl}(N)$}-link homology {$(N\geq 4)$} using foams and the
  {K}apustin--{L}i formula.
\newblock {\em Geom. Topol.}, 13(2):1075--1128, 2009.
\newblock \href {http://arxiv.org/abs/0708.2228} {\path{arXiv:0708.2228}},
  \href {https://doi.org/10.2140/gt.2009.13.1075}
  {\path{doi:10.2140/gt.2009.13.1075}}.

\bibitem[MV07]{MacVaz}
M.~Mackaay and P.~Vaz.
\newblock The universal {${\rm sl}_3$}-link homology.
\newblock {\em Algebr. Geom. Topol.}, 7:1135--1169, 2007.
\newblock \href {http://arxiv.org/abs/math/0603307}
  {\path{arXiv:math/0603307}}, \href {https://doi.org/10.2140/agt.2007.7.1135}
  {\path{doi:10.2140/agt.2007.7.1135}}.

\bibitem[QR16]{QueRos}
H.~Queffelec and D.~E.~V. Rose.
\newblock The {$\mathfrak{sl}_n$} foam 2-category: a combinatorial formulation
  of {K}hovanov--{R}ozansky homology via categorical skew {H}owe duality.
\newblock {\em Adv. Math.}, 302:1251--1339, 2016.
\newblock \href {http://arxiv.org/abs/1405.5920} {\path{arXiv:1405.5920}},
  \href {https://doi.org/10.1016/j.aim.2016.07.027}
  {\path{doi:10.1016/j.aim.2016.07.027}}.

\bibitem[RW20]{RobWag}
L.-H. Robert and E.~Wagner.
\newblock A closed formula for the evaluation of foams.
\newblock {\em Quantum Topol.}, 11(3):411--487, 2020.
\newblock \href {http://arxiv.org/abs/1702.04140} {\path{arXiv:1702.04140}},
  \href {https://doi.org/10.4171/qt/139} {\path{doi:10.4171/qt/139}}.

\bibitem[SS06]{SeidelSmith}
P.~Seidel and I.~Smith.
\newblock A link invariant from the symplectic geometry of nilpotent slices.
\newblock {\em Duke Math. J.}, 134(3):453--514, 2006.
\newblock \href {http://arxiv.org/abs/math/0405089}
  {\path{arXiv:math/0405089}}, \href
  {https://doi.org/10.1215/S0012-7094-06-13432-4}
  {\path{doi:10.1215/S0012-7094-06-13432-4}}.

\bibitem[SS22]{StroppelSussan}
C.~Stroppel and J.~Sussan.
\newblock A {L}ie theoretic categorification of the coloured {J}ones
  polynomial.
\newblock {\em J. Pure Appl. Algebra}, 226(10):Paper No. 107043, 43, 2022.
\newblock \href {http://arxiv.org/abs/2109.12889} {\path{arXiv:2109.12889}},
  \href {https://doi.org/10.1016/j.jpaa.2022.107043}
  {\path{doi:10.1016/j.jpaa.2022.107043}}.

\bibitem[Sus07]{SussanThesis}
J.~Sussan.
\newblock {\em Category {O} and sl(k) link invariants}.
\newblock ProQuest LLC, Ann Arbor, MI, 2007.
\newblock Thesis (Ph.D.)--Yale University.
\newblock URL: \url{https://www.proquest.com/docview/304773953}, \href
  {http://arxiv.org/abs/math/0701045} {\path{arXiv:math/0701045}}.

\bibitem[Tut47]{Tutte1947}
W.~T. Tutte.
\newblock A ring in graph theory.
\newblock {\em Proc. Cambridge Philos. Soc.}, 43:26--40, 1947.
\newblock \href {https://doi.org/10.1017/s0305004100023173}
  {\path{doi:10.1017/s0305004100023173}}.

\bibitem[Web17]{Webster}
B.~Webster.
\newblock Knot invariants and higher representation theory.
\newblock {\em Mem. Amer. Math. Soc.}, 250(1191):v+141, 2017.
\newblock \href {http://arxiv.org/abs/1309.3796} {\path{arXiv:1309.3796}},
  \href {https://doi.org/10.1090/memo/1191} {\path{doi:10.1090/memo/1191}}.

\bibitem[Wit17]{Witten}
E.~Witten.
\newblock Two lectures on the {J}ones polynomial and {K}hovanov homology.
\newblock In {\em Lectures on geometry}, Clay Lect. Notes, pages 1--27. Oxford
  Univ. Press, Oxford, 2017.
\newblock \href {http://arxiv.org/abs/1401.6996} {\path{arXiv:1401.6996}}.

\bibitem[Wu14]{Wu}
H.~Wu.
\newblock A colored {$\germ {sl}(N)$} homology for links in {$S^3$}.
\newblock {\em Dissertationes Math.}, 499:1--217, 2014.
\newblock \href {http://arxiv.org/abs/0907.0695} {\path{arXiv:0907.0695}},
  \href {https://doi.org/10.4064/dm499-0-1} {\path{doi:10.4064/dm499-0-1}}.

\bibitem[Yon11]{Yonezawa}
Y.~Yonezawa.
\newblock Quantum {$(\germ{sl}_n,\wedge V_n)$} link invariant and matrix
  factorizations.
\newblock {\em Nagoya Math. J.}, 204:69--123, 2011.
\newblock \href {http://arxiv.org/abs/0906.0220} {\path{arXiv:0906.0220}},
  \href {https://doi.org/10.1215/00277630-1431840}
  {\path{doi:10.1215/00277630-1431840}}.

\end{thebibliography}

\end{document}